\documentclass[journal]{IEEEtran}
\usepackage{cite}
\usepackage{url}
\usepackage{empheq}
\usepackage{float}
\usepackage{dsfont}
\usepackage{amsthm}
\usepackage{algorithm}
\usepackage{algpseudocode}

\usepackage{amsmath, amssymb, amsfonts}
\allowdisplaybreaks[4]
\usepackage{graphicx}
\usepackage{textcomp}
\usepackage{enumerate}
\usepackage{epstopdf}
\usepackage{array}
\usepackage{booktabs}
\usepackage{subfigure}
\usepackage{multirow}
\usepackage{soul, color}
\usepackage[dvipsnames]{xcolor}
\usepackage[latin1]{inputenc}

\newtheorem{lemma}{Lemma}
\newtheorem{corollary}{Corollary}
\newtheorem{proposition}{Proposition}

\soulregister\cite{7} % Õë¶Ô\citeÃüÁî
\soulregister\citep7 % Õë¶Ô\citepÃüÁî
\soulregister\citet7 % Õë¶Ô\citetÃüÁî
\soulregister\ref{7} % Õë¶Ô\refÃüÁî
\soulregister\pageref{7} % Õë¶Ô\pagerefÃüÁî

\usepackage{nomencl}
\makenomenclature
\usepackage{etoolbox}
\renewcommand{\nomgroup}[1]{%
\item[\bfseries \ifstrequal{#1}{A}{Acronyms}{%
      \ifstrequal{#1}{S}{Symbols}{%
    \ifstrequal{#1}{U}{Units}{}}}%
]}

\begin{document}
\title{Forecast-Enhanced Lyapunov Optimization for Real-Time EV Charging Scheduling}

\author{ Shihan Huang,
Yue Chen,~\IEEEmembership{Senior Member,~IEEE},
Richard Chen,
Adam Wierman
%Steven H. Low,~\IEEEmembership{Fellow,~IEEE}
%\thanks{This work was supported by the National Natural Science Foundation of China under Grant No. 52307144. (Corresponding to Y. Chen)}
%\thanks{S. Huang was with the Department of Mechanical and Automation Engineering, The Chinese University of Hong Kong, Hong Kong, China, and is now with China Resources Power Technology Research Institute Co., Ltd., Shenzhen, China  (e-mail: shhuang@link.cuhk.edu.hk).}
\thanks{S. Huang and Y. Chen are with the Department of Mechanical and Automation Engineering, The Chinese University of Hong Kong, Hong Kong, China (e-mail: shhuang@link.cuhk.edu.hk, yuechen@mae.cuhk.edu.hk).}
\thanks{R. Chen and A. Wierman are with the Computing and Mathematical Sciences Department, California Institute of Technology, Pasadena, CA 91125 USA (e-mail: rachen@caltech.edu, adamw@caltech.edu).}
}
\markboth{Journal of \LaTeX\ Class Files,~Vol.~XX, No.~X, Oct.~2026}%
{Shell \MakeLowercase{\textit{et al.}}: Bare Demo of IEEEtran.cls for
IEEE Journals}

\maketitle

\begin{abstract}
Electric vehicles (EVs) play a vital role in achieving carbon neutrality. Various approaches have been developed for online optimal EV charging scheduling to maximize their environmental and economic benefits. Among them, Lyapunov optimization has gained wide adoption due to its ease of implementation, no need for predictions, and rigorous performance guarantees. However, this prediction-free nature also limits the performance of Lyapunov optimization, as it cannot fully leverage the relatively accurate short-term forecasts often available in practice. To overcome this limitation, this paper proposes a forecast-enhanced Lyapunov optimization method for real-time EV charging scheduling. Specifically, we design novel virtual queues and embed the traditional Lyapunov optimization within a receding horizon control framework to incorporate short-term predictions. The proposed algorithm is further extended by introducing heterogeneous penalty parameters to reduce the optimality gap. We prove that the proposed algorithm achieves bounded charging delay and a bounded optimality gap between online and offline solutions, both depending on the prediction window length. Numerical experiments demonstrate that the proposed method reduces operational costs compared to the traditional prediction-free Lyapunov optimization algorithm, while still satisfying all charging requirements. 
\end{abstract}

% Note that keywords are not normally used for peerreview papers.

\IEEEpeerreviewmaketitle

\section*{Nomenclature}
\addcontentsline{toc}{section}{Nomenclature}

% \subsection{Acronyms}
% \begin{IEEEdescription} [\IEEEusemathlabelsep
% \IEEEsetlabelwidth{$E_v^{\mathrm{int}}$, $E_v^{\mathrm{dep}}$}]
% \item [EV] {Electric vehicle.}
% \item [MPC] {Model predictive control.}
% \item [SOC] {State of charge.}
% \item [RL] {Reinforcement learning.}
% % \item [LSTM] {Long short-term memory}
% % \item [GRU] {Gated recurrent unit}
% % \item [CNN] {Convolutional neural network}
% \end{IEEEdescription}

\subsection{Sets}
\begin{IEEEdescription}[\IEEEusemathlabelsep
\IEEEsetlabelwidth{$E_v^{\mathrm{int}}$, $E_v^{\mathrm{dep}}$}]
\item [$\mathcal{S}$] {Set of all EVs.}
\item [$\mathcal{S}_g$] {Set of all EVs of group $g$.}
\item [$\mathcal{S}_g^{\mathrm{act}}$] {Set of active EVs of group $g$.}
\item [$\mathcal{U}_v$] {Set of time slots when EV $v$ is available for charging.}
\end{IEEEdescription}

\subsection{Parameters}
\begin{IEEEdescription}[\IEEEusemathlabelsep
\IEEEsetlabelwidth{$E_v^{\mathrm{int}}$, $E_v^{\mathrm{dep}}$}]
\item [$a_v, a_g$] {Arrival rate of charging tasks of EV $v$ and group $g$.}
\item [$D_g$] {Maximum charging delay of EVs in group $g$.}
% \item [$E_v^{\mathrm{a}}, E_v^{\mathrm{d}}$] {Energy levels of EV $v$ upon arrival and departure.}
\item [$E_v^{\mathrm{req}}$] {Required energy for charging of EV $v$.}
\item [$E_v^{\max}$] {Maximum energy that EV $v$ can be charged.}
\item [$E_v^{\mathrm{int}}$, $E_v^{\mathrm{dep}}$] {Initial energy level and the desired energy level at departure of EV $v$.}
\item [$E_v^{\mathrm{cap}}$] {Battery capacity of EV $v$.}
\item [$G$] {Number of groups of EVs.}
\item [$P_v$] {Maximum charging power of EV $v$.}
% \item[$\underline{X}_g^{\mathcal{U}}, \overline{X}_g^{\mathcal{U}}$] {Minimum and maximum charging energy of group $g$ during time horizon $\mathcal{U}$.}
\item [$Q_g, Z_g$] {Maximum backlogs of group $g$'s virtual queues.}
\item [$R_g$] {Maximum allowable charging delay of EVs in group $g$.}
\item [$T$] {Number of time slots and the duration of each time slot is $\Delta t$.}
\item [$T_v^{\mathrm{a}}, T_v^{\mathrm{d}}$] {Arrival and departure time of EV $v$.}
\item [$V, V_g$] {Homogeneous and heterogeneous penalty parameters in Lyapunov optimization.}
\item [$w$] {Length of the lookahead window.}
\item [$\alpha_g$] {Weight in delay-aware virtual queue.}
\item [$\eta$] {Charging efficiency.}
\item [$\pi$] {Electricity price at time $t$.}
\end{IEEEdescription}

\subsection{Variables}
\begin{IEEEdescription}[\IEEEusemathlabelsep \IEEEsetlabelwidth{$E_v^{\mathrm{int}}$, $E_v^{\mathrm{dep}}$}]
% \item [$e_v, e_g$] {Energy level of EV $v$ and group $g$.}
\item [$p_v$] {Charging power of EV $v$ at time $t$.}
\item [$q_g, z_g$] {Virtual queues of group $g$ at time $t$.}
\item [$x_g$] {Total charging power of group $g$ at time $t$.}
\end{IEEEdescription}

\section{Introduction}

As a major pathway to decarbonization, the electrification of transportation has been accelerated. % by the synergy of regulatory mandates and technological progress, leading to a rapid increase in the global stock of electric vehicles (EVs). 
The global market share of EVs has exceeded 20\% in 2024 according to the International Energy Agency \cite{Global2025}.
However, the growing EV charging demand, if left unmanaged, poses significant challenges to power system operation, potentially causing network congestion or even cascading failures \cite{LiImpact2024}. Effective EV charging scheduling is therefore essential, yet it remains challenging due to numerous underlying uncertainties, such as random EV arrivals, variable renewable generation, and volatile electricity prices \cite{NimalsiriSurvey2020}. Since accurate forecasts of these uncertainties over the entire time horizon are rarely available in practice, the development of an online algorithm for real-time EV charging scheduling becomes indispensable.
Existing online optimization algorithms for real-time EV charging scheduling can be broadly categorized into two classes: \emph{prediction-based} ones and \emph{prediction-free} ones.
Model predictive control (MPC) is a classic prediction-based approach that
uses an explicit system model and forecasts of uncertainty at each time step  to solve an optimization problem over a finite lookahead window.
MPC has been applied in EV-related research on  charging scheduling \cite{SanerChargeCurveBattery2023}, energy management \cite{WangSequentialClusteringMethod2024},  power flexibility \cite{ShaoSchedulingEVCharging2023}, etc.
A robust MPC approach was developed in \cite{YangRobust2023} to  avoid missing service by one-step worst-case prediction.
A block MPC approach was developed in \cite{YangEV2024} to minimize the demand charge on the scheduling of deferrable EV charging demand.
However, MPC suffers from two limitations: First, its performance is dependent on the accuracy of the multi-step forecasts and substantial prediction errors can significantly degrade performance. Second, MPC can be computationally intensive and may not be scalable for a large system as the dimensions of states and inputs grow.

To overcome the shortcomings of prediction-based approaches (e.g., MPC), prediction-free methods, which do not rely on explicit forecasts of future uncertainties, have  been  explored as promising alternatives. Reinforcement learning (RL) is one such example. 
An RL agent learns a control policy by exploration and exploitation guided by a reward signal.
RL offers a model-free approach to real-time scheduling, such as aggregate flexibility feedback \cite{LiLearningBased2021} and pricing schemes \cite{AdetunjiTwoTailed2024}.
RL was used to solve an NP-hard charging station deployment problem in \cite{ZhaoOptimalEVFast2023}.
A two-stage learning-based framework was proposed in \cite{ZhangEV2024}, in which RL was used to allocate the commands of the grid. The competition among charging stations was modeled as a non-cooperative game, and the optimal pricing strategy was derived by a multi-agent RL framework in \cite{QianMultiAgent2022}.
Despite their potential, RL-based methods require vast amounts of training data and computational resources. Moreover, RL may not be suitable for problems with hard constraints, such as EV charging scheduling problems, due to the lack of safety guarantees.

Apart from learning-based approaches, Lyapunov optimization is also a powerful prediction-free optimization framework.
It offers rigorous performance guarantees and low computational complexity \cite{NeelyStochastic2010}, and thus is particularly suitable for real-time EV charging scheduling. 
A deadline-aware charging scheduling scheme was proposed based on  Lyapunov optimization  in \cite{WangOnline2025} considering state-of-charge (SOC)-dependent peak charging rate.
Real-time pricing and resource allocation of EV charging stations were optimized by Lyapunov optimization in \cite{LuRealTimePricingAllocation2024}.
Energy sharing between charging stations was investigated in \cite{YanDistributed2023}, and a decentralized algorithm was developed for privacy protection.
A real-time aggregate EV power flexibility characterization method was developed in \cite{YanRealTime2024}.
A demand response program was designed in \cite{MENGHWAR2024122426} to alleviate network congestion, in which EV charging demand and flexible household load were scheduled by Lyapunov optimization.
Lyapunov optimization was combined with a game-theoretic approach to jointly optimize pricing schemes and charging schedules of charging stations in \cite{AbbasiCoupled2024}.

% | Methodology | Core Principle | Key Strengths | Critical Weaknesses
% | Dependencies |
% | :---- | :---- | :---- | :---- | :---- |
% | Lyapunov Optimization | Per-slot drift-plus-penalty
% minimization | Robust to uncertainty; Low complexity; Provable
% performance guarantees | Myopic decisions; Suboptimal long-term
% performance | Current system state only |
% | Model Predictive Control | Receding-horizon optimization |
% Anticipatory control; Near-optimal performance with good forecasts |
% High computational complexity; Sensitivity to forecast errors |
% System model; Multi-step forecasts |
% | Reinforcement Learning | Trial-and-error policy learning (MDP)
% | Model-free; Handles complex non-linearities | High sample
% complexity; Slow convergence; Lack of performance guarantees | Large
% historical or simulated dataset for training |

While Lyapunov optimization has performed pretty well in EV charging scheduling, we notice an opportunity to further enhance its performance. That is, to incorporate the relatively accurate short-term forecasts of EV arrivals, renewable power supply, and electricity prices, which will give rise to the forecast-enhanced Lyapunov optimization in this paper. With the rapid development of short-term forecasting techniques \cite{TawnReview2022}, uncertainties in EV charging scheduling, such as EV arrival time, electricity prices and renewable generation, can be predicted with reasonable accuracy for a very short horizon (e.g., from several minutes to a few hours). For example, a two-step learning framework was developed to forecast the half-hourly step electricity prices in \cite{GhimireTwostep2024}, with an error of only 2.12\% in predicting the peak electricity price.
Wind power has been forecasted by machine learning techniques such as Reformer \cite{WangMPCbased2025} or an ensemble model \cite{WangEndtoEnd2025}, with a mean square error of 0.004 p.u. in second-scale prediction and a mean absolute percentage error of 3.7\% in hour-scale prediction respectively. Incorporating these accurate short-term forecasts into the Lyapunov optimization framework could enhance its performance, which has not been explored.
Virtual queues have been adopted in recent learning-based lookahead approaches, such as \cite{guoSafeOptimalOnline2025}, where the algorithm learns  reward and cost models of each action.
However, this method relies on handcraft design of the reward and cost structures and lacks theoretical performance guarantees.
%However, the performance of such approaches relies heavily on the reward and cost structures, while the proposed method adopts forecast exogenous system inputs such as EV arrivals and electricity prices, which can be acquired with high accuracy in practice.
%Moreover, there is no  strict safety gurarantee by MPC or online learning, i.e., the charging delay in the EV charging problem, which can be derived analytically in the proposed framework. 

To fill the research gaps above, this paper proposes a novel approach that incorporates high-accuracy  short-term forecasts into the Lyapunov optimization framework, which mitigates its myopic decision-making while retaining the theoretical performance guarantees. 
The key idea is  to incorporate short-term forecasts into the virtual-queue construction. Specifically, we construct lookahead virtual queues over a prediction window and embed the  drift-plus-penalty algorithm into a receding-horizon optimization framework.
In summary, our main contributions are two-fold:

\begin{enumerate}[1)]
  \item We develop a novel \emph{forecast-enhanced Lyapunov optimization method} that leverages accurate short-term predictions to improve online decision-making.  
  New virtual queue designs are proposed to capture near-future charging demand over a look-ahead window.
  Compared with the traditional Lyapunov optimization, the proposed method is less myopic and significantly reduces total charging costs while satisfying EV charging requirements. It also outperforms the MPC method both in terms of charging cost and delay. To the best of our knowledge, this is the first work that incorporates forecasts into the Lyapunov optimization framework.
  \item We provide \emph{theoretical performance guarantees} for the proposed algorithm. Specifically, we prove that the charging delay of all EVs is upper-bounded by a term related to the length of the prediction window. Moreover, the cost gap between the offline optimal solution and the online decisions by the proposed algorithm is also proven to be bounded. Case studies further illustrate the trade-off between reduced charging delay and lower charging costs, as well as the key factors influencing this balance.
\end{enumerate}

The remainder of this paper is organized as follows. EV charging scheduling is modeled in Section \ref{sec:model}. The forecast-enhanced Lyapunov optimization framework is developed in Section \ref{sec:lya}. The proposed algorithm is then validated and its performance is evaluated by case studies in Section \ref{sec:result}. Finally, conclusions are drawn in Section \ref{sec:conclu}.

\section{EV Charging Scheduling Model}
\label{sec:model}
In this section, we first introduce the offline EV charging scheduling problem and then present its relaxed formulations for developing the online algorithm in Section \ref{sec:lya}.

\subsection{Offline EV Charging Scheduling Problem}

Let $T$ and $\Delta t$ be the number of time slots and the duration of each time slot, respectively. 
Denote the set of all EVs by $\mathcal{S}$ and the battery charging efficiency by $\eta$.
For each EV $v \in \mathcal{S}$, let $T_{v}^{\mathrm{a}}$ and $T_{v}^{\mathrm{d}}$ be its arrival and departure time respectively, with $T_{v}^{\mathrm{a}} < T_{v}^{\mathrm{d}}$. Then, the available window of charging service for each EV is defined by 
\begin{equation}
  \mathcal{U}_v \triangleq \left\{ T_v^{\mathrm{a}}, \dots, T_v^{\mathrm{d}}-1 \right\}.
\end{equation}
Note that an EV is allowed to be charged at its arrival time slot but not at its departure time slot. 
Let $E_v^{\mathrm{int}}$ and $E_v^{\mathrm{dep}} $ be the  battery energy level of EV $v$ when it arrives and leaves respectively. Denote the battery capacity of the EV $v$ by $E_v^{\mathrm{cap}}$. Then we  define $E_{v}^{\mathrm{req}} \triangleq E_v^{\mathrm{dep}}-E_v^{\mathrm{int}}$ and $E_{v}^{\max} \triangleq E_v^{\mathrm{cap}}-E_v^{\mathrm{int}}$, which are the required and maximum   charging energy of EV $v$, respectively. %Assume $0 < E_{v}^{\mathrm{req}} \le E_{v}^{\max}$ for every $v \in \mathcal{S}$, and 
Let $P_v$ be its maximum  charging power. 
Let $f(t)$ be the total charging cost of all EVs at time $t$, given by
\begin{equation}
  f(t) \triangleq \pi(t) \sum_{v \in \mathcal{S}} p_{v}(t),
\end{equation}
where $\pi(t)$ is the electricity price and $p_{v}(t)$ is the charging power of EV $v$ at time slot $t$.
The optimal charging scheduling problem of  EVs is formulated by
\begin{subequations}
\begin{align}
  \textbf{P1:} \inf_{\substack{p_v(t), \\ \forall v \in \mathcal{S}, \\ \forall t =1, \dots, T}}
  & \lim_{T \to \infty} \frac{1}{T} \sum_{t=1}^{T} \mathbb{E} \left[ f(t) \right], \label{obj:ev} \\
  \mbox{s.t.} \quad
  & 0 \le p_{v}(t) \le P_{v}, ~ \forall t \in \mathcal{U}_v, \forall v \in \mathcal{S},\label{ineq:power_bd}      \\
  & p_{v}(t) =0, ~ \forall t \notin \mathcal{U}_v, \forall v \in \mathcal{S},
  \label{eq:power_mask}      \\
  & E_{v}^{\mathrm{req}} \le \eta \Delta t \sum_{t=1}^{T} p_{v}(t) \le E_{v}^{\max}, \forall v \in \mathcal{S}. \label{ineq:energy_req}
\end{align}
\end{subequations}
% \textit{All constraints in \textbf{P1} hold for every $v \in \mathcal{S}$.}
The objective function \eqref{obj:ev} aims to minimize the average total charging cost. Constraints \eqref{ineq:power_bd} and \eqref{eq:power_mask} ensure that EVs charge only when they are in the charging station. Their charging requirements are met as given in \eqref{ineq:energy_req}. Moreover, the right-hand side of \eqref{ineq:energy_req} also ensures that the EV battery energy level always remains within its capacity limit, i.e., $ E_v^{\mathrm{int}} + \Delta t \sum_{t=1}^{\tau} p_{v}(t) \le  E_v^{\mathrm{int}} + \Delta t \sum_{t=1}^{T} p_{v}(t) \le  E_{v}^{\mathrm{cap}},\forall \tau =1, \cdots, T,\forall v \in \mathcal{S}$.

\subsection{Relaxed Reformulation}

\textbf{P1} is  not suitable for real-time implementation since the summation in \eqref{obj:ev} and \eqref{ineq:energy_req} taken over the entire scheduling horizon requires future information such as EV arrivals and electricity prices.
To address this issue, we first relax \textbf{P1} into an equivalent aggregate EV charging scheduling problem to make it suitable for online algorithm design in Section \ref{sec:lya}.

%need to transform \textbf{P1} into an equivalent aggregate EV charging scheduling problem, and then relax the problem to make it suitable for online implementation.

In practice, EVs can be grouped by the EV aggregator based on certain criteria, such as their maximum charging power, location, battery capacity (if available), etc., by which \textbf{P1} is transformed into an aggregate charging scheduling problem. The benefit of grouping EVs is that the number of decision variables is reduced significantly, and thus the problem can be solved more efficiently for real-time implementation. In this study, the EVs are grouped by their parking time for the convenience of defining virtual queues later. Denote the parking time of the group $g$ by $R_g$, i.e., $| \mathcal{U}_v | = R_g$, $\forall v \in \mathcal{S}_g$.

Suppose the EVs are divided into $G$ groups, indexed by $g=1,\dots,G$. Denote the set of EVs in the group $g$ by $\mathcal{S}_g$. Formally we have
\begin{subequations}
\begin{align}
  & \mathcal{S} = \bigcup_{g=1}^{G} \mathcal{S}_g, \\
  & \mathcal{S}_i \cap \mathcal{S}_{j} = \emptyset, ~ \forall i \neq j, ~ i,j \in \{1, \dots, G\}.
\end{align} 
\end{subequations}
Moreover, define a subset of $\mathcal{S}_g$ for the convenience of subsequent formulation:    
\begin{align}
  \mathcal{S}_g^{\mathrm{act}} (t) &\triangleq \left\{v \in \mathcal{S}_g: t \in \mathcal{U}_v \right\}, 
\end{align}
where $\mathcal{S}_g^{\mathrm{act}} (t)$ is the set of  active EVs (i.e., available for charging) at time $t$.
For group $g$ at time  $t$, the aggregate  charging power $x_{g}(t)$ is given  by
\begin{align}
  x_{g}(t) \triangleq \sum_{v \in \mathcal{S}_g} p_{v} (t).
\end{align}
In the aggregate charging scheduling problem, the objective function $f(t)$ can reformulated as
\begin{equation}
  f(t) = \pi(t) \sum_{g=1}^{G} x_{g}(t).
\end{equation}

We also define the aggregate EV charging demand:
\begin{gather}\label{eq:agvt-lb}
    a_v(t)\! = \!\!\left\{\begin{array}{ll}
    P_v,     \!\! & \!\! \mathrm{if}~ T_v^{\mathrm{a}} \leq t \le T_v^{\mathrm{a}} + T_v^{\min}, \\
    \frac{E_v^{\mathrm{req}}}{\eta} - T_v^{\min} P_v, \!\!& \!\!\mathrm{if}~ t = T_v^{\mathrm{a}} + T_v^{\min}+1, \\
    0,           \!\! & \!\! \mathrm{otherwise},
    \end{array}\right.
\end{gather}
where $T_v^{\min} = \left\lfloor \frac{E_v^{\mathrm{req}}}{P_v \eta} \right\rfloor$. 
Let $a_g(t)=\sum_{v \in \mathcal{S}_g} a_v(t)$.
% For each EV $v$, its   EV charging demand at each time slot is evaluated based on the two assumptions:
% \begin{enumerate}[(1)]
%   \item Each EV can start charging as soon as it arrives at the charging station,
%     and it will be charged until its energy level reaches the desired level before departure.
%   \item Each EV can be charged at its maximum charging power $P_{v}$ during    the charging period except for the last time slot of charging.
% \end{enumerate}
% The details of EV charging model can be found in \cite{HuangOptimal2025}. 
% In a word, the charging demand of an EV is equal to its maximum charging power during the charging period except for the last time slot of charging, and the aggregate arrival rate of charging demand $a_{g}(t)$  is given by the sum of the charging demand of all EVs in group $g$ at time slot $t$.

Assume $\pi(t)$, $T_v^{\mathrm{a}}$ and $E_v^{\mathrm{req}}$ to be i.i.d. across time or EVs, then the aggregate EV charging scheduling problem can be defined based on $a_{g}(t)$:
\begin{subequations}
\begin{align}
  \textbf{P2:} \inf_{\substack{x_g(t), \\ \forall g =1, \dots, G, \\ \forall t =1, \dots, T}}
  & \lim_{T \to \infty} \frac{1}{T} \sum_{t=1}^{T} \mathbb{E} \left[ f(t) \right] \label{obj:agg_exp} ,\\
  \hbox{s.t.} \quad       
  & 0 \le x_{g}(t) \le \!\! \sum_{v \in \mathcal{S}_g^{\mathrm{act}} (t)} \min \left\{ \frac{E_{v}^{\max}}{\eta \Delta t}, P_v \right\},  \nonumber    \\
  & \forall t, \forall g, \label{ineq:agg_power_bd}  \\
  & \lim_{T\to\infty} \frac{1}{T} \sum_{t=1}^{T} \mathbb{E}[a_{g}(t) - x_{g}(t)] \leq 0, \forall g. \label{ineq:ag_xg}    
\end{align}
\end{subequations}
The expectations in \eqref{obj:agg_exp} and \eqref{ineq:ag_xg} are taken over the distribution of $\pi(t)$ and $a_{g}(t)$, respectively.
Constraints \eqref{ineq:agg_power_bd} and \eqref{ineq:ag_xg} guarantee that all charging demand can be satisfied in the long run. 
\textbf{P2} is  a relaxation of \textbf{P1}, by which the energy requirement is satisfied over an infinite horizon instead of a finite horizon. 

\section{Forecast-Enhanced Lyapunov Optimization Framework}
\label{sec:lya}

In this section, we present the proposed forecast-enhanced Lyapunov optimization framework to derive an online algorithm for solving \textbf{P2}.

\subsection{Introducing Forecasts to Virtual Queues}
\label{subsec:lya_forecast}

The standard Lyapunov optimization framework is \emph{prediction-free} \cite{HuangOptimal2025}. In this study, we propose to leverage the relatively accurate short-term forecasts of charging demands and electricity prices to enhance the performance of the Lyapunov optimization framework, while preserving its theoretical guarantees.

Suppose we have the accurate predictions for charging  demand and electricity  prices in $w$ time slots, where  $w$ is an integer in the range $[1,T]$. 
First, we propose a new virtual queue design to integrate forecasts:
\begin{equation}
    \label{eq:Qgt}
    q_{g}(t+w) = \max \left\{q_{g}(t) + \sum_{\tau=0}^{w-1} (a_{g}(t+\tau) - x_{g}(t+\tau)), 0 \right\},
\end{equation}
where $q_{g}(t)$ is the backlog of charging demands at time slot $t$. Unserved
charging demands expected to arrive between $t$ and $t+w$ will be stacked in
$q_{g}(t+w)$ for processing.

Moreover, the delay-aware virtual queue $z_{g}(t)$ is defined by
\begin{equation}
    \label{eq:zgt}
    z_{g}(t+w) = \max \left\{z_{g}(t) - \sum_{\tau=0}^{w-1}x_{g}(t+\tau) + \frac{\alpha_{g}}{R_{g}} \mathbb{I}_{g}(t), 0 \right\} ,
\end{equation}
where $\alpha_{g}$ is a positive number, $R_{g}$ is the parking time of each group and $\mathbb{I}_{g}(t)$ is the
indicator function of $q_{g}(t)$, given by
\begin{align}
\mathbb{I}_{g}(t) \triangleq \left\{
  \begin{array}{ll}0, & \mathrm{if}~ q_g(t) = 0,\\ 1, &
    \mathrm{if}~ q_g(t) > 0.
  \end{array} \right.
\end{align}
$z_{g}(t)$ is used to control the charging delay of EVs since it keeps growing as long as there is unserved charging demand.
Moreover, the initial conditions of $q_{g}(t)$ and $z_{g}(t)$ are given by
\begin{equation}    
\label{eq:queue_init}
    q_g(0) = z_g(0) = 0,  \quad \forall g = 1, 2, \dots, G.
\end{equation}

The virtual queues $q_{g}(t)$ and $z_{g}(t)$ are processed by the
first-in-first-out
(FIFO) method, that is, the charging demand $a_{g}(t)$ is served
in the order    of $t$.
In this way, the maximum charging delay of group $g$ is bounded    by the following    proposition.
\begin{proposition}[Maximum Charging Delay]
  \label{prp:delay}
  Suppose the virtual queues $q_{g}(t)$ and $z_{g}( t)$ are bounded, i.e.,
  $q_{g}(g) \le Q_{g}, z_{g}(t) \le Z_{g}, \forall t$, then the maximum
  charging delay $D_{g}$ of group $g$ is no more than
  \begin{equation}
    \label{ineq:delay_bound}
    D_{g}< \frac{w R_{g}(Q_{g}+ Z_{g})}{\alpha_{g}}        .
  \end{equation}
\end{proposition}
The proof of Proposition \ref{prp:delay} can be found in Appendix    \ref{appd:delay}.
If the maximum charging delay of group $g$  $D_g \le R_g$, then every charging demand generated by an EV in  group $g$ can be served before it leaves the charging station under FIFO service. 
In this sense, Proposition \ref{prp:delay} bridges the gap between the relaxed aggregate formulation and individual-level feasibility. 
Note that the framework is reduced to the standard Lyapunov optimization   when $w=1$.
Hence, it is revealed by \eqref{ineq:delay_bound} that incorporating the predictions may lead to increased delays. 
Thus, $w$ and $\alpha_g$ should be carefully selected to guarantee feasibility.

Based on the virtual queues, we propose $\textbf{P2}'$, a relaxed counterpart of \textbf{P2} that integrates predictions.
%\textbf{P2} can be further transformed into the following problem by $q_{g}(t)$:
\begin{subequations}    
\begin{align}
  \textbf{P2$'$:} \inf_{x_g(t), \forall g, \forall t}~
  & \lim_{T \to \infty} \frac{1}{T} \sum_{t=1}^{T} \mathbb{E} \left[ \frac{1}{w} \sum_{\tau=0}^{w-1} f(t+\tau) \right], \label{eq:obj-P3'}  \\
  \hbox{s.t.}~        
  % & \eqref{eq:eg_init}, \nonumber \\
  & \eqref{ineq:agg_power_bd}, \label{ineq:agg_power_bd-2} \\
  & \lim_{T\rightarrow\infty} \frac{1}{T} \mathbb{E} [q_{g}(T)] = 0.  \label{eq:mean_rate_stable}
\end{align}
\end{subequations}
Note that \eqref{obj:agg_exp} is replaced by the average cost  over the lookahead window in $\textbf{P2}'$.
\eqref{eq:mean_rate_stable} is called mean-rate stability of $q_{g}(t)$. It can be shown that any feasible solution of $\textbf{P2}'$ is also a feasible solution to \textbf{P2}. Specifically, since $q_g(t)=0$ for  $t=1,2,...,w$, then the following inequality can be derived
from \eqref{eq:Qgt}:
\begin{align}
  \label{eq-1}q_{g}(t+w) - q_{g}(t) \ge
  \sum_{\tau=0}^{w-1}a_{g}(t+\tau) -
  \sum_{\tau=0}^{w-1}x_{g}(t+\tau).
\end{align}
Summing \eqref{eq-1} up over $t=1, 1+w,...,1+w (T-1)$ and
dividing both sides
by $w T$, we have
\begin{equation}
  \label{ineq:ag_xg_1}\lim_{T\rightarrow\infty}\frac{1}{w
  T}\sum_{t=1}^{w T}\mathbb{E}
  [a_{g}(t) - x_{g}(t)] \le \lim_{T\rightarrow\infty}\frac{1}{w T}\mathbb{E}
  [q_{g}(w T+1)] = 0.
\end{equation}
Hence, \eqref{ineq:ag_xg} holds if \eqref{eq:mean_rate_stable} holds.

\subsection{Drift-Plus-Penalty Based Transformation}
\label{subsec:drift_plus_penalty}
Since $\textbf{P2}'$ still contains the summation over time in the objective function, it is also an offline problem. To address this challenge, we use a \textit{drift-plus-penalty} based transformation to turn it into an online counterpart $\textbf{P3}$.

%Further, the Lyapunov optimization framework is applied to transform $\textbf{P2}'$ and derive an online solution algorithm.
%The core idea of Lyapunov optimization is to minimize a \textit{drift-plus-penalty} term at each time slot to balance the cost minimization (objective of $\textbf{P2}'$) and queue stability (constraint \eqref{eq:mean_rate_stable} of $\textbf{P2}'$).

%instead of $f(t)$.
%By minimizing the drift-plus-penalty term, the algorithm aims to balances the trade-off between queue stability and cost minimization. 
First,  define the Lyapunov function as  the sum of the square of     backlogs of all virtual queues:
\begin{equation}
  \label{eq:LyaFun}
  L(\boldsymbol{\Theta}(t)) \triangleq \frac{1}{2} \sum_{g = 1}^{G}
  \left( q_{g}^{2}(t) + z_{g}^{2}(t) \right),
\end{equation}
where $\boldsymbol{\Theta}(t)$ is the system state vector at time slot $t$, $\boldsymbol{\Theta}(t) = (q_{1}(t), \dots, q_{G}(t), z_{1}(t),
\dots, z_{G}(t))$.
In other words, the Lyapunov function is  a scalar measure of the stability of a dynamic system. Then the Lyapunov drift from time slot $t$ to time slot $t+w$ is defined as
\begin{equation}
  \label{eq:LyaDrift}
  % \begin{aligned}
    \Delta(\boldsymbol{\Theta}(t))  \triangleq \mathbb{E}
    [L(\boldsymbol{\Theta}(t+w)) - L(\boldsymbol{\Theta}(t)) | \boldsymbol{\Theta}(t)].
  % \end{aligned}
\end{equation}
$\Delta(\boldsymbol{\Theta}(t))$ indicates the expected increment in the total backlogs of virtual queues. Further, we define the drift-plus-penalty term:
%To minimize the objective function of $\textbf{P2}'$ while ensuring the stability of virtual queues, \eqref{obj:agg_exp} should be replaced by the following drift-plus-penalty term:
\begin{equation}
\label{obj:drift_penalty}
    \frac{1}{w} \Delta(\boldsymbol{\Theta}(t)) + V \mathbb{E} \left[\frac{1}{w} \sum_{\tau=0}^{w-1} f(t+\tau) \Bigg| \boldsymbol{\Theta}(t) \right].
\end{equation}
The first term in \eqref{obj:drift_penalty} is called the Lyapunov drift from time slot $t$ to time slot $t+w$, which represents the increment in the Lyapunov function over the lookahead window of length $w$. 
The second term, i.e., the penalty term, is  the original objective function \eqref{eq:obj-P3'} weighted by a non-negative parameter $V$, which is a parameter to be determined.
%However, the term $\Delta(\boldsymbol{\Theta}(t))$ in \eqref{obj:drift_penalty} is still time-coupled, which cannot be solved online.
%it is impossible to solve a problem with the drift-plus-penalty term \eqref{obj:drift_penalty} in an online manner due to the term $\Delta(\boldsymbol{\Theta}(t))$.
%To address this issue, an upper bound of the drift-plus-penalty term is derived. Specifically, we have
Furthermore, we derive an  upper bound of the drift-plus-penalty term 
\begin{equation}
  \label{eq:LyaDriftExpan}
  \begin{aligned}
    \Delta(\boldsymbol{\Theta}(t)) = ~ &L(\boldsymbol{\Theta}(t+w)) - L(\boldsymbol{\Theta}(t)) \\
   = ~& \frac{1}{2}\sum_{g = 1}^{G} \left(q_{g}^{2}(t+w) -
    q_{g}^{2}(t)       
     + z_{g}^{2}(t + w) -
    z_{g}^{2}(t) \right) \\
    \le~ & B - \sum_{g = 1}^{G} \left((q_{g}(t) + z_{g}(t)) \sum_{\tau=0}^{w-1}x_{g}(t + \tau) \right),
  \end{aligned}
\end{equation}
where 
\begin{equation}
    B = \frac{1}{2} \sum \nolimits_{g = 1}^{G} \left( B_{g1} + 2w Q_{g} A_{g} +  B_{g2} \right).
\end{equation}

The inequality \eqref{eq:LyaDriftExpan} is proven below:
\begin{proof}
Assume the arrival of charging demand is bounded above in a single time slot, i.e., $a_{g}(t) \le A_{g}, \forall t = 1, \ldots, T$.
%, where $A_{g}$ is determined by the capacity of the charging service provided by the aggregator, otherwise \eqref{ineq:ag_xg} will not be satisfied.
Since the charging power at every time slot $t$ is also bounded, the following inequality holds for any time slot $t$:
\begin{equation}
    \left( \sum \nolimits_{\tau=0}^{w-1} (x_{g}(t + \tau) - a_{g}(t + \tau) \right)^{2} \le B_{g1},
\end{equation}
which means that the change in the backlog of $q_{g} (t)$ during the lookahead window has a unified bound $B_{g1}$. 
Therefore, the first term in the summation of \eqref{eq:LyaDriftExpan} can be bounded as follows according to \eqref{eq:Qgt}:
\begin{equation}
  \label{ineq:qg_sq_ub}
  \begin{aligned}
    q_{g}^{2}(t+w) & \le \left(q_{g}(t)-
      \sum_{\tau=0}^{w-1}x_{g}(t+\tau) +
    \sum_{\tau=0}^{w-1}a_{g}(t+\tau) \right)^{2}   \\
    & \le q_{g}^{2}(t) +
    \left(\sum_{\tau=0}^{w-1}x_{g}(t+\tau)-\sum_{\tau=0}^{w-1}a_{g}(t+\tau)\right)^{2}
    \\
    & \quad + 2q_{g}(t) \left(\sum_{\tau=0}^{w-1}a_{g}(t+\tau) -
    \sum_{\tau=0}^{w-1}x_{g}(t+\tau) \right)               \\
    & \le  q_{g}^{2}(t) + B_{g1} + 2w Q_{g} A_{g}- 2q_{g}(t)
    \sum_{\tau=0}^{w-1}x_{g}(t+\tau).
  \end{aligned}
\end{equation}
Similarly, we have
\begin{equation}
\label{ineq:zg_sq_ub}
  \begin{aligned}
    z_{g}^{2}(t+w) & \le \left(z_{g}(t) +
      \frac{\alpha_g}{R_g}\mathbb{I}_{g}(t) -
    \sum_{\tau=0}^{w-1}x_{g}(t+\tau) \right)^{2} \\
    &  \le z_{g}^{2}(t) + B_{g2} - 2z_{g}(t) \sum_{\tau=0}^{w-1}x_{g}(t + \tau), 
  \end{aligned}
\end{equation}
where $B_{g2}$ is a constant given by
\begin{equation}
    B_{g2} = \left( \frac{\alpha_g}{R_g}\right)^{2}+ (w X_{g})^{2}+ 2Z_{g}\frac{\alpha_g}{R_g}.
\end{equation}
Therefore, \eqref{eq:LyaDriftExpan} holds, which completes the proof.
\end{proof}

Ignoring the constant terms, the following online optimization
problem can be formulated by replacing $\Delta(\boldsymbol{\Theta}(t))$ in \eqref{obj:drift_penalty} with the right-hand side of \eqref{eq:LyaDriftExpan}:
\begin{subequations}
  \begin{align}
    \textbf{P3:} \!\! \min_{\substack{x_g(t+\tau), \\ \forall g =1, \dots, G, \\ \forall \tau = 0, \dots, w-1}} ~ & V \sum_{\tau=0}^{w-1} f(t+\tau) \nonumber
    \\
    & - \sum_{g=1}^{G}\left( \left(q_{g}(t) + z_{g}(t)\right)
    \sum_{\tau=0}^{w-1}x_{g}(t+\tau)\right) \label{obj:online} \\
    \hbox{s.t.} \quad 
    &  0 \le x_{g}(t+\tau) \le \sum_{v \in \mathcal{S}_g^{\mathrm{act}} (t+\tau)} \!\! \min \left\{ P_v, \frac{E_{v}^{\max}}{\eta \Delta t} \right\} , \nonumber \\ 
    & \forall \tau = 0, \dots, w-1, \forall g = 1, \dots, G. \label{ineq:agg_power_bd_win}
  \end{align}
\end{subequations}

% \textbf{P3} is an online optimization problem with a lookahead    window of $w$ time slots. 
The summation over $t \in \{1, \dots, T\}$ in \textbf{P2} and $\textbf{P2}'$ no longer exists in \textbf{P3}, thus \textbf{P3} can be solved in an online manner. \textbf{P3} is a linear programming problem that can be solved efficiently by standard optimization solvers. Denote the optimal solution of \textbf{P3} at time $t$ by $\hat
    {x}_{g}^{t}(t+\tau),\forall \tau=0,\cdots,w-1$.

\subsection{Overall Online Algorithm}

% {\color{red} add the description of the algorithm here; mention the buffered average policy in Algorithm 1 and update the citation in Proposition 4, Corollary 1, and Proposition 5.}
Noting that for a specific time $t$, we have different optimal decisions obtained by \textbf{P3} at time $t, t-1, \cdots, t-w+1$, respectively, i.e., $\hat {x}_{g}^{t-\tau}(t),\forall \tau = 0, \cdots, w-1$. To integrate these decisions, we solve \textbf{P3} within a receding horizon control framework, and the policy implemented  at time $t$ is the buffered average of the policies derived by \textbf{P3} from time $t-w+1$ to $t$, i.e., 
%Now we propose a buffered average policy leveraging the solutions of \textbf{P3}.  
%The policy implemented    at time $t$ is the buffered average of the policies derived by \textbf{P3} from time $t-w+1$ to $t$, i.e., 
\begin{equation}
\label{eq:buffered_average}
  x_g (t) = \frac{1}{w}        \sum_{\tau=0}^{w-1} \hat{x}_{g}^{(t-\tau)}(t), ~ \forall g = 1, \dots, G.
\end{equation}
%The policy \eqref{eq:buffered_average} has desirable properties such as bounded optimality gap and charging delay, which will be revealed by the following propositions.% The  framework of the proposed algorithm with the buffered average policy is presented in Algorithm xx.

The overall algorithm is presented in Algorithm 1. In the following, we prove that the charging decisions obtained by Algorithm 1 have theoretical performance guarantees.

\begin{algorithm}[t]
\caption{Forecast-Enhanced Lyapunov Optimization with Buffered Average Policy}
\begin{algorithmic}[1]
\State \textbf{Input:} Price $\pi (t)$, $t = 1, \dots, T$, EV parameters $P_v, E_v^{\mathrm{req}}, E_v^{\max}, T_{v}^{\mathrm{a}}, T_{v}^{\mathrm{d}}$, $v \in \mathcal{S}$
\State \textbf{Initialize:} $w$, $V$, $\alpha_g$ and virtual queues $q_g (1) = z_g (1) = 0$ for every $g = 1, \dots, G$, policy buffer $\mathcal{B} = \emptyset$
\For{$t = 1$ to $T - w + 1$}
    \If{$|\mathcal{B}| = w$} remove the oldest policy from $\mathcal{B}$ 
    \EndIf
    \State Solve \textbf{P3} with $\pi (t+\tau)$, $\tau = 0, \dots, w-1$, get 
    \[\boldsymbol{X} (t) = \left[ \hat{x}_g^{(t)} (t+\tau)  \right]_{g = 1, \dots, G}^{\tau = 0, \dots, w-1} \]
    \State $\mathcal{B} \gets \mathcal{B} \cup \{\boldsymbol{X} (t)\}$, and compute $x_g (t)$ with $\mathcal{B}$ by \eqref{eq:buffered_average}
    \State Update $q_g(t)$ and $z_g(t)$  by \eqref{eq:Qgt} and \eqref{eq:zgt} respectively
    \State Disaggregate $x_g (t)$ to $p_v (t), v \in \mathcal{S}_g$ for every $g = 1, \dots, G$ by FIFO method
\EndFor
\State \textbf{Output:} Power allocation $[p_v (t)]_{v \in \mathcal{S}}^{t = 1, \dots, T}$
\end{algorithmic}
\end{algorithm}

%Another advantage of the proposed framework is that     the backlogs of virtual queues $q_{g}(t)$ and $z_{g}(t)$ can be controlled, which guarantees the maximum charging delay according to Proposition \ref{prp:delay}. To derive the relation between the parameter $V$ and the charging delay, the following lemma is first presented.

{
\setlength{\lineskiplimit}{3pt}
\setlength{\lineskip}{3pt}
In Proposition \ref{prp:delay}, we have proven that all EVs can meet their charging requirements with bounded delay, which shows the feasibility of the proposed algorithm. Further, we discuss its optimality. Let $\left( x_g^{*}(t) \right)_{g=1,\dots,G}$ be the optimal solution to \textbf{P2} at time $t$, and let $\left( \tilde{x}_g (t) \right)_{g=1,\dots,G}$ be the policy given by \eqref{eq:buffered_average} at time $t$. Moreover, assume $\hat{f}^{(\tau)}(t) = 0$, $\forall \tau \le 0$, $\forall t$.
Then we have the following proposition:
\par
}

% {
% \setlength{\lineskiplimit}{3pt}
% \setlength{\lineskip}{3pt}
% The utilization of future information is expected to enhance the performance of the Lyapunov optimization framework. Specifically, the gap between the performance of \textbf{P2} and that of \textbf{P3} can be bounded if proper charging scheduling strategy is adopted by the aggregator.
% Let $\left\{ x_g^{*}(t): g=1,\dots,G \right\}$ be the optimal solution to \textbf{P2} at time $t$, and let $\left\{ \hat{x}_g^{(t)}(t+\tau): \tau = 0, \dots, w-1; g = 1, \dots, G  \right\}$ be the optimal solution to \textbf{P3} at time $t$.
% Let $f^* (t) = \sum_{g=1}^{G} \pi (t) x_g^{*} (t)$, $\hat{f}^{(t)} (t+\tau) =  \sum_{g=1}^{G} \pi (t+\tau) \hat{x}_g^{(t)} (t+\tau)$.
% Then we have the following proposition:
% \par
% }
\begin{proposition}[Optimality Gap with Forecast]
  \label{prp:optimality}
  The optimality gap between \textbf{P2} and \textbf{P3} is bounded under the buffered average policy \eqref{eq:buffered_average}, i.e.,
  \begin{equation}
    \label{ineq:gap}
   0 \le  f^{\text{ONL}}- f^{\text{OFL}}        \le  \frac{B}{w V},
  \end{equation}
  where $f^{\text{OFL}}$ and $f^{\text{ONL}}$ are the total cost achieved by \textbf{P2} and \textbf{P3} respectively, given by
  \begin{subequations}
  \begin{align}
      f^{\text{OFL}} &= \lim_{T\to \infty} \frac{1}{T} \sum_{t=1}^{T}    \sum_{g=1}^{G} \pi (t) x_g^{*} (t) , \\
      f^{\text{ONL}} 
      &= \lim_{T\to \infty} \frac{1}{T} \sum_{t=1}^{T}    \mathbb{E} \left[ \left. \sum_{g=1}^{G} \pi (t) \tilde{x}_g (t) \right|  \boldsymbol{\Theta}(t) \right]. \nonumber \\
      &= \lim_{T\to \infty} \frac{1}{T} \sum_{t=1}^{T}  \mathbb{E} \left[ \left. \frac{1}{w} \sum_{\tau=0}^{w-1} \hat{f}^{(t-\tau)}(t) \right|  \boldsymbol{\Theta}(t)\right].
  \end{align}
  \end{subequations}
\end{proposition}

The proof of Proposition \ref{prp:optimality} can be found in Appendix \ref{appd:optimality}. It is revealed by \eqref{ineq:gap} that the performance gap decreases as the lookahead window $w$ increases, which demonstrates the benefit of utilizing future information. However, a larger $w$ may lead to a larger maximum charging delay according to \eqref{ineq:delay_bound}, which indicates a trade-off between the cost minimization and charging delay. Hence, the length of the lookahead window should be carefully chosen in practice.

\begin{proposition}[Uniform Bounds of Virtual Queues]
\label{prp:queue_bound}
Under the buffered average policy \eqref{eq:buffered_average}, for any $t \in [1, 2, \dots, T]$, there exists an upper bound of $q_g(t)$ and $z_g(t)$ respectively, given by
\begin{align}
    q_g(t) & \le Q_g = V\pi_{\max} + 2wX_g, \label{ineq:qg_bound} \\
    z_g(t) & \le Z_g = V\pi_{\max} + \frac{2 \alpha_g}{R_g}, \label{ineq:zg_bound}
\end{align}
where $X_g = \sum_{v \in \mathcal{S}_g} P_{v}$.
\end{proposition}

The proof of Proposition \ref{prp:queue_bound} can be found in Appendix \ref{appd:queue_bound}.
Both $q_g(t)$ and $z_g(t)$ are uniformly bounded for all time slots by Proposition \ref{prp:queue_bound}. As a result, the maximum charging delay in terms of $V$ and $\alpha$ can be derived based on \eqref{ineq:delay_bound}, \eqref{ineq:qg_bound} and \eqref{ineq:zg_bound}.

\begin{corollary}
\label{cor:delay}
Under the buffered average policy \eqref{eq:buffered_average}, the maximum charging delay of group $g$ is bounded by
\begin{equation}
    D_g < \frac{2 w R_g}{\alpha_g} \left( V \pi_{\max} + w X_g + \frac{\alpha_g}{R_g} \right).
\end{equation}
\end{corollary}

\subsection{Generalization of \textbf{P3}}

Note that the penalty parameter $V$ in \eqref{obj:drift_penalty} is a scalar. Alternatively, heterogeneous penalty parameters can be used for different groups, by which the performance of the proposed framework may be further improved. For simplicity, collect $V_g$, $q_g(t)$, $z_g(t)$ and $x_g(t)$ of each group $g$ into vectors $\boldsymbol{V}$, $\boldsymbol{q}(t)$, $\boldsymbol{z}(t)$ and $\boldsymbol{x}(t)$ respectively. Then \textbf{P3} can be reformulated as:
\begin{align}
    \textbf{P3$'$:} \min_{\substack{\boldsymbol{x}(t+\tau), \\ \forall \tau = 0, \dots, w-1}} 
    & \sum_{\tau=0}^{w-1} \left( \pi(t+\tau) \boldsymbol{V} - \boldsymbol{q}(t) - \boldsymbol{z}(t) \right)^\top \boldsymbol{x}(t+\tau) ,\\
    % \right. \nonumber \\ 
    % & \left. - (\boldsymbol{q}(t) + \boldsymbol{z}(t))^\top \boldsymbol{x}(t+\tau) \right]. \\
    \hbox{s.t.} \quad & \eqref{ineq:agg_power_bd_win}. \nonumber
\end{align}

\begin{proposition}[Optimality Gap with Heterogeneous Penalty]
  \label{prp:optimality_alt}
  If the policy implemented    at time $t$ is given by \eqref{eq:buffered_average},  then the gap between \textbf{P2} and $\textbf{P3}'$ is bounded by
  \begin{equation}
    \label{ineq:gap_alt}
        f^{\text{ONL}}- f^{\text{OFL}} \le  \frac{B}{w V_{\min}},
  \end{equation}
  where $V_{\min} = \min_g \{V_g: g = 1, \dots, G\}$.
\end{proposition}

The proof of Proposition \ref{prp:optimality_alt} can be found in Appendix \ref{appd:optimality_alt}.
Proposition \ref{prp:optimality_alt} indicates that the performance gap between \textbf{P2} and $\textbf{P3}'$ is determined by the minimum penalty parameter among all groups. The bound given by \eqref{ineq:gap_alt} will be reduced to that of \eqref{ineq:gap} if $V_g = V, \forall g$.    

To achieve the best performance under the heterogeneous penalty setting, $V_g$ should be chosen as the value just enough to complete the charging task for each group $g$. In other words, any $V' > V_g$ cannot satisfy the charging demand of group $g$. The optimal choices of $V_g$ with different lookahead window lengths will be discussed in Section \ref{sec:result}. 

\section{Case Studies}
\label{sec:result}

The proposed algorithm is validated and its
performance is evaluated
by numerical simulations based on real-world data.

\subsection{Simulation Setup}

Case studies are conducted by a desktop computer with Intel Core i5-10505 with 16 GB RAM, implemented in MATLAB. Two cases are considered in the simulation:
\begin{itemize}
  \item \textbf{Case 1: slow charging.} The maximum charging power of EVs is 7--30 kW, which is typical for residential charging. The total number of EVs is set to 200.
  \item \textbf{Case 2: fast charging.} The maximum charging power of EVs is  100--400 kW, which is typical for fast charging stations. The total number of EVs is set to 100.
\end{itemize}
The arriving time of EVs are set based on the real-world data of EVs in Shanghai \cite{DuImpact2025}. A typical scenario of a working day is considered, where most EVs arrive at the charging station between 8:00 and 10:00 am and their parking time ranges from 5 to 11 hours. The simulation horizon is set to 24 hours with a time slot length of 5 minutes. 
The charging efficiency is set to 0.95. The initial SOCs of EVs range from 25\% to 80\%, and the target SOC of each EV is 80\%, i.e., the SOC of each EV should be no less than 80\% before it leaves the charging station. Their parking time is set sufficiently large, which guarantees that charging demand can be satisfied at its maximum charging power within their parking duration.
Moreover, their parking time is assumed to be totally heterogeneous in this case, i.e., each EV is a single group itself, which is the worst case for the proposed algorithm. 
The electricity price is obtained from the California Independent System Operator (CAISO) \cite{Grid2025}.

The proposed algorithm with homogeneous and heterogeneous penalty parameters are implemented respectively, and the performance of the proposed algorithm is compared with the following benchmark algorithms:
\begin{itemize}
    \item \textbf{Greedy algorithm}: At each time slot, the charging power of each group is set to its maximum value, i.e., the upper bound given by \eqref{ineq:agg_power_bd}.
    % \item \textbf{Earliest deadline first (EDF)}: At each time slot,  the available charging power is allocated to EVs with earlier deadlines first.
    % \item \textbf{Least laxity first (LLF)}: At each time slot, the available charging power is allocated to EVs with smaller laxity first, where the laxity of an EV is defined as the difference between its remaining parking time and the time required to complete its charging task, given by:
    % \begin{align}
    %     \text{Laxity}_v (t) = T_v^{\mathrm{d}} - t - \frac{E_v^{\mathrm{req}} - e_v (t)}{\eta P_v}, \forall v \in \mathcal{S}.
    % \end{align}
    \item \textbf{MPC}: Problem \textbf{P4} (introduced below) is solved at each time slot.     Denote the optimal solution of \textbf{P4} by $\left[ \hat{p}_v (t+\tau) \right]_{v \in \mathcal{S}}^{\tau = 0, \dots, w-1}$, then only the policy at the first step $(\hat{p}_v(t))_{v \in \mathcal{S}}$ is implemented.
    \item \textbf{Standard Lyapunov optimization}: The standard Lyapunov optimization framework is solving \textbf{P3} with $w=1$, i.e., without utilizing future information.
    \item \textbf{Offline algorithm}: The offline optimal solution is obtained by solving \textbf{P1} with perfect information over the entire horizon.
    \item \textbf{Proposed algorithm with heterogeneous penalty (A1)}: The proposed forecast-enhanced Lyapunov optimization framework is implemented with heterogeneous penalty parameters, i.e., solving $\textbf{P3}'$ with $w=12$.
    \item \textbf{Proposed algorithm with homogeneous penalty (A2)}: The proposed forecast-enhanced Lyapunov optimization framework is implemented with a homogeneous penalty parameter, i.e., solving \textbf{P3}  with $w=12$.
\end{itemize}
%{\color{blue}Note that A1 and A2 share the same length of lookahead window with the MPC method. A2 reduces to the standard Lyapunov optimization method when $w=1$.} 

Specifically, the MPC problem is given by
\begin{subequations}
\begin{align}
  \textbf{P4:} & \min_{\substack{p_v(t+\tau), \\\forall v \in \mathcal{S}, \\ \forall \tau = 0, \dots, w-1}} ~  \frac{1}{w} \sum_{\tau=0}^{w-1} f(t+\tau), \\
  \hbox{s.t.} \quad 
  & 0 \le p_{v}(t+\tau) \le P_{v}, \forall t+\tau \in \mathcal{U}_v, \forall v \in \mathcal{S}, \\
  & p_{v}(t+\tau) =0, \forall t+\tau \notin \mathcal{U}_v, \forall v \in \mathcal{S},          \\
  & e_v (1) = E_v^{\mathrm{a}}, \forall v \in \mathcal{S}, \\
  & e_{v}(t+\tau+1) = e_{v}(t+\tau) + \eta p_{v}(t+\tau) \Delta t,   \\
  & \forall \tau = 0, \dots, w-1, \forall v \in \mathcal{S}, \\
  & e_{v}(t+w-1) \ge e_{v}^{\mathrm{term}}(t), \forall v \in \mathcal{S}^{\mathrm{vis}}(t), \label{ineq:mpc_term} \\
  & e_{v}(t+\tau) \le E_{v}^{\mathrm{cap}},  \forall \tau = 0, \dots, w-1, \forall v \in \mathcal{S}, \\
  & e_{v} \left(T_v^{\mathrm{d}}\right) \ge E_{v}^{\mathrm{req}},   \\
  & \text{if } t \le T_v^{\mathrm{d}} \le t+w-1, \forall v \in \mathcal{S},
\end{align}
\end{subequations}
where $e_v (t)$ is the energy level of EV $v$ at time slot $t$.
The terminal constraint \eqref{ineq:mpc_term} is imposed to ensure that the charging demand of each EV can be satisfied within its parking duration, where $\mathcal{S}^{\mathrm{vis}}(t)$ is the set of EVs that are present in the charging station during the lookahead window, given by
\begin{equation}
    \mathcal{S}^{\mathrm{vis}}(t) = \left\{ v \in \mathcal{S}: t + w \ge  T_v^{\mathrm{a}}, t < T_v^{\mathrm{d}} \right\}.
\end{equation}
$e_{v}^{\mathrm{term}}(t)$ is the minimum energy level that EV $v$ should reach at the end of the lookahead window to satisfy its charging demand, given by
\begin{equation}
    e_{v}^{\mathrm{term}}(t) =     
    \max\left\{
        E_{v}^{\mathrm{a}},
        E_{v}^{\mathrm{req}} - \eta \Delta t P_{v} R^{\mathrm{after}}_{v}(t)
    \right\},
\end{equation}
and
$R^{\mathrm{after}}_{v}(t)$ is the remaining time slots after the lookahead window until the departure of EV $v$, given by
\begin{equation}
    R^{\mathrm{after}}_{v}(t) = \max\left\{0, T_v^{\mathrm{d}} - t - w \right\}.
\end{equation}
%\eqref{ineq:mpc_term} requires  the assumption that the departure time of each EV is also known when it enters the lookahead window. The MPC-based algorithm would tend to be inactive before the arrivals of charging demand enters the lookahead window without such information, thus charging task would fail for some EVs. 
%Since this assumption does not necessarily hold in reality,  \textbf{P6} serves as only a benchmark for comparison. In fact, the proposed algorithm can still have  better performance than MPC even without this information, as will be shown in the simulation results.

\subsection{Validating and Benchmarking the Proposed Algorithm}

The maximum charging delay of several example groups with different lengths of lookahead window are presented in Table \ref{tab:delay}.  The results are derived by A2 with $V$ fixed at 10.
%The theoretical bounds of maximum charging delay derived in \eqref{ineq:delay_bound} are also presented in Table \ref{tab:delay} for comparison. 
The maximum charging delay of all EVs under different lengths of lookahead window are all lower than the theoretical bounds given by Proposition \ref{prp:delay}, which demonstrates the effectiveness of the proposed framework in controlling the charging delay of EVs.
%The average maximum charging delay of all EVs under different lengths of lookahead window are also presented in the second-to-last column of Table \ref{tab:delay} for reference.

\begin{figure}[htbp]
    \centering
    \subfigure[]
    {\includegraphics[width=0.5\linewidth]{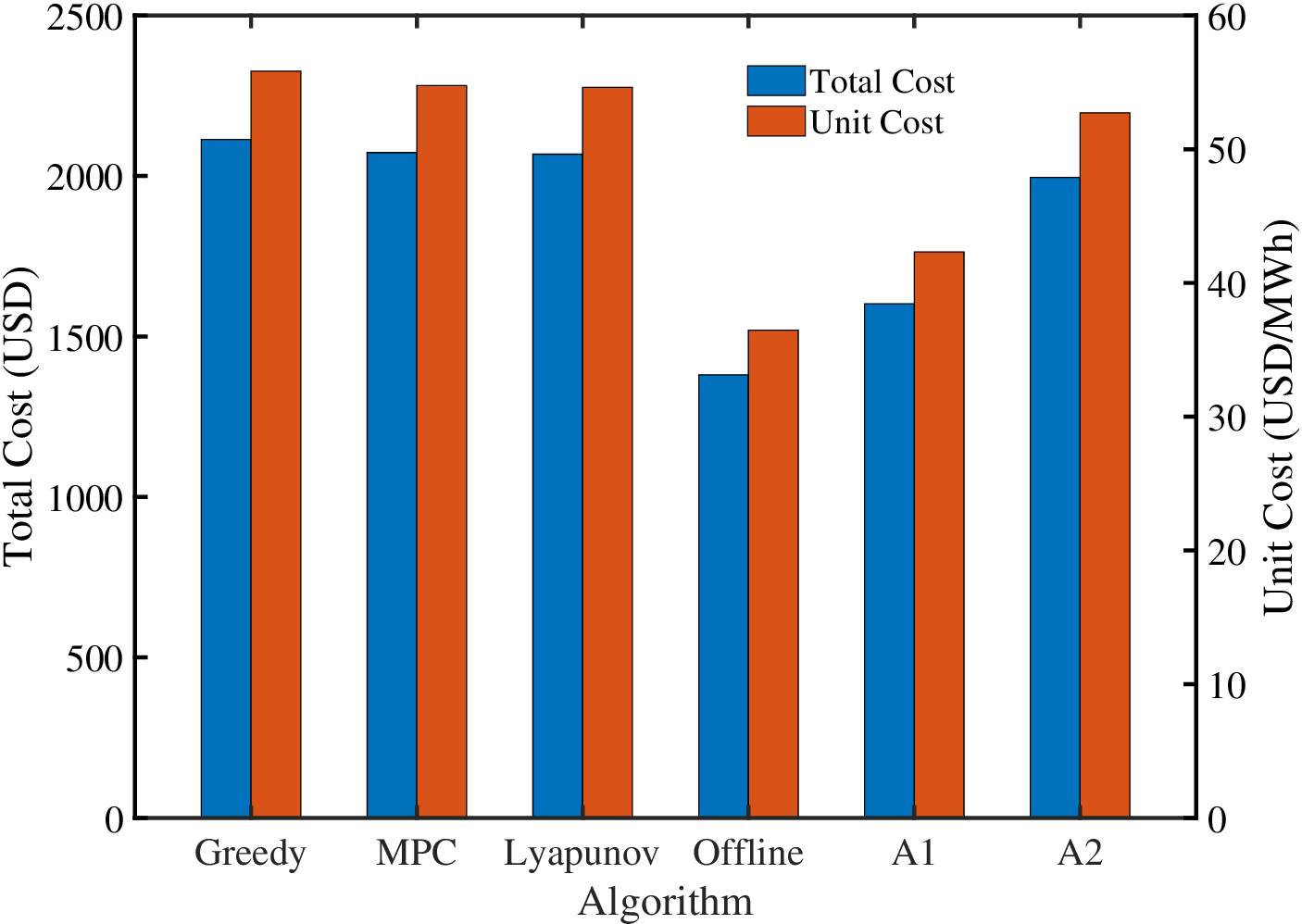} \label{fig:cost_slow}}
    \subfigure[]
    {\includegraphics[width=0.45\linewidth]{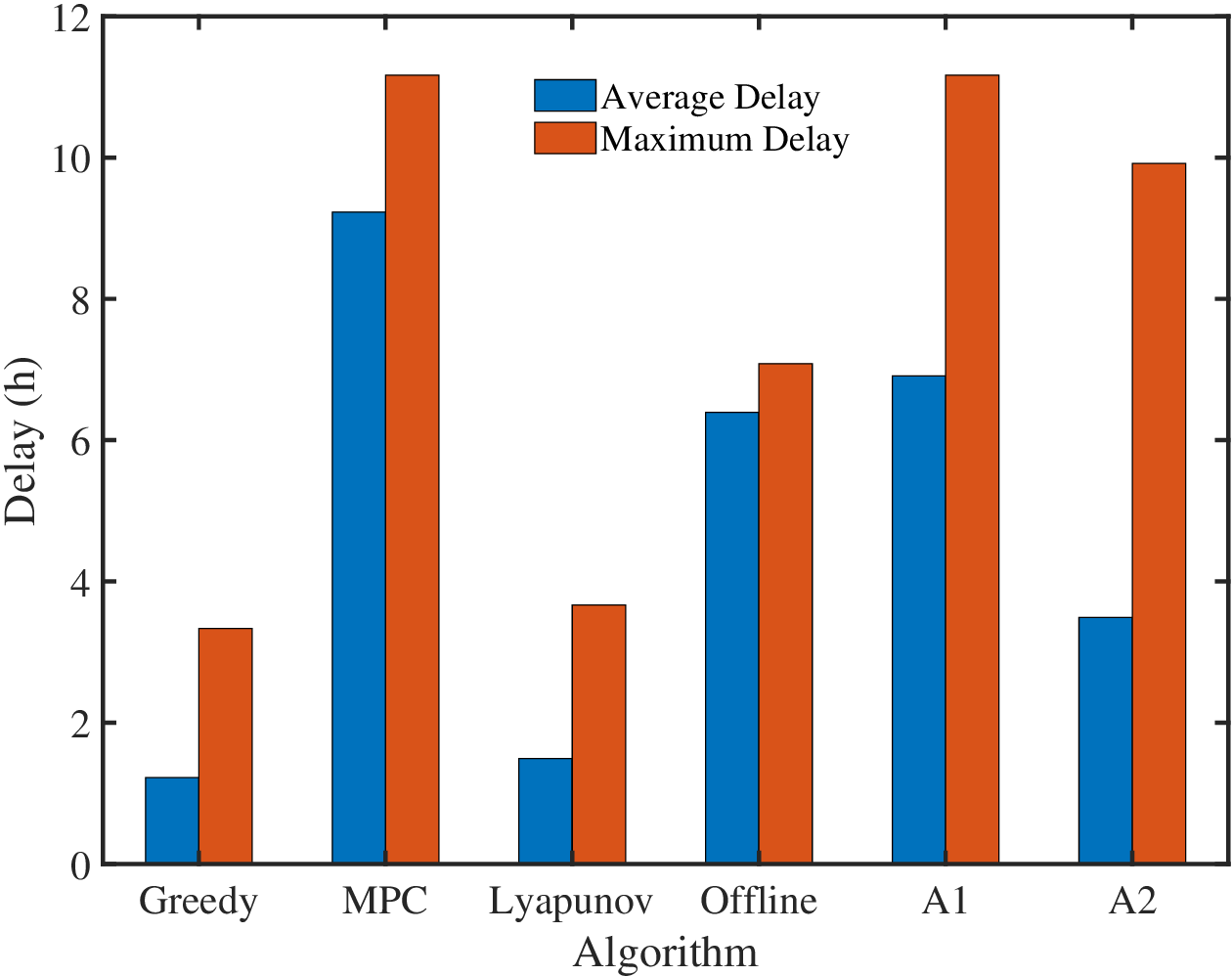} \label{fig:delay_slow}}
    \subfigure[]
    {\includegraphics[width=0.5\linewidth]{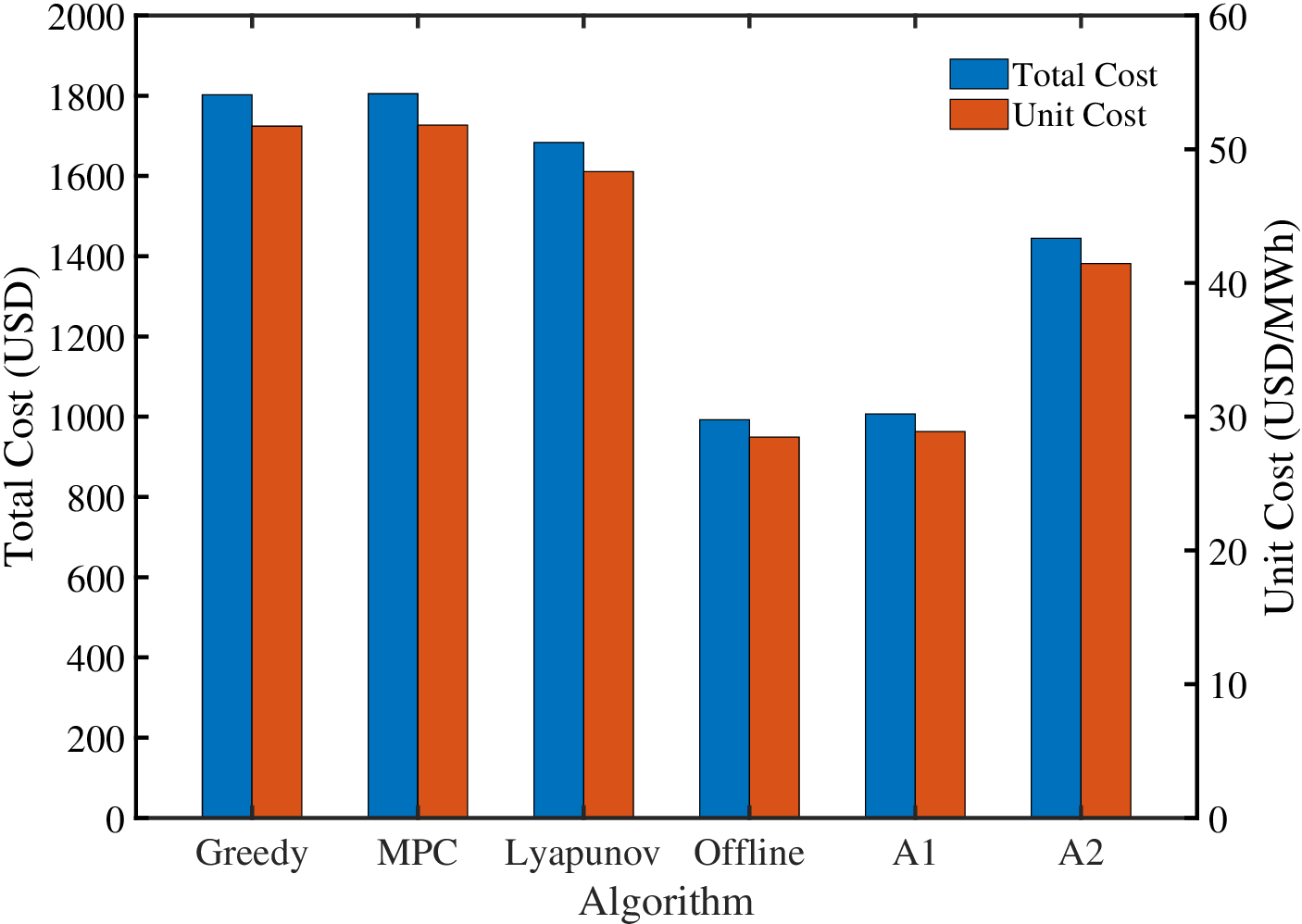} \label{fig:cost}}
    \subfigure[]
    {\includegraphics[width=0.45\linewidth]{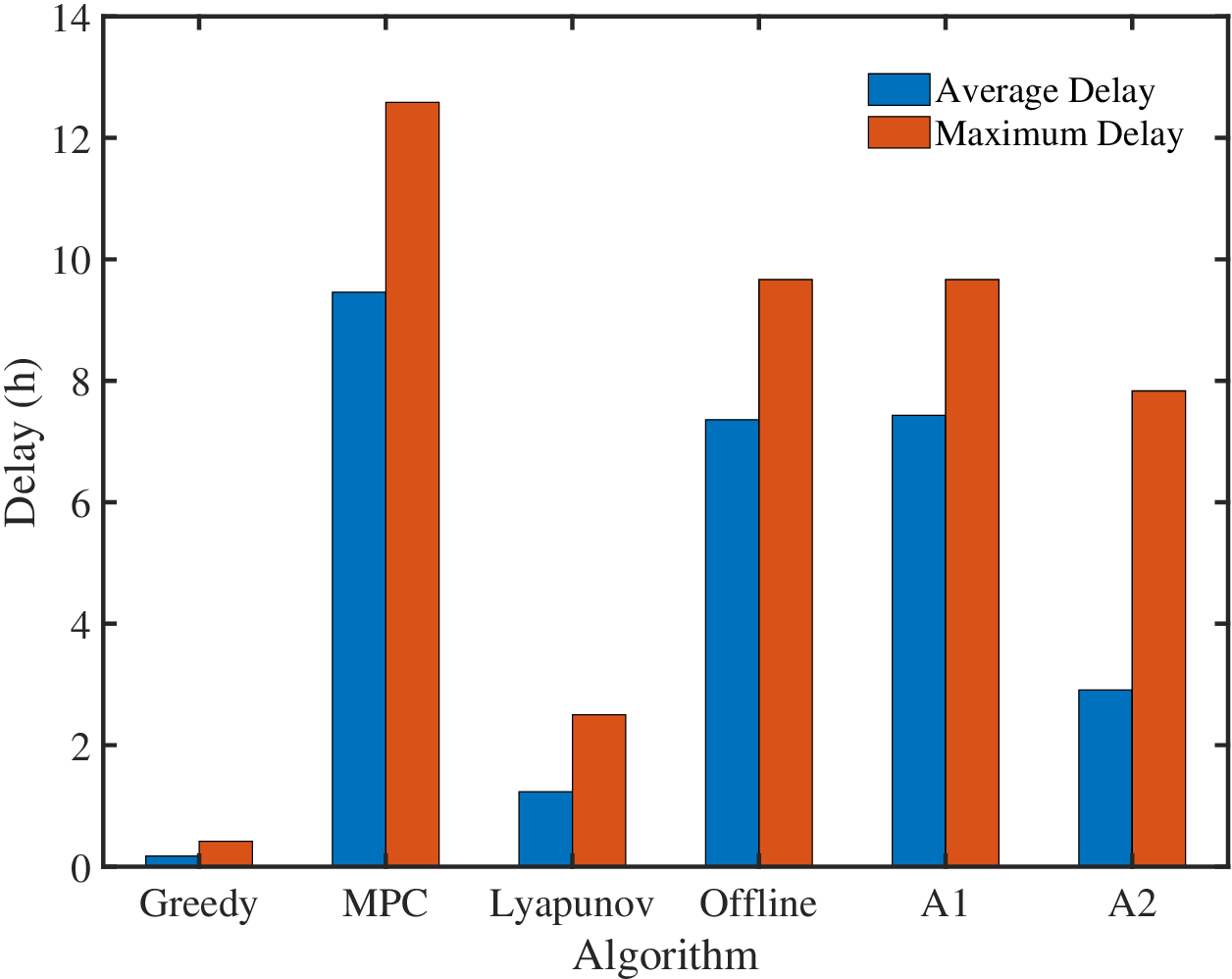} \label{fig:delay}}
    \caption{Comparison of total cost and maximum charging delay of all EVs under different algorithms. (a)(b) Slow charging case; (c)(d) Fast charging case.}
    \label{fig:compare}
\end{figure}

\begin{table}[t]
    \centering
    \caption{Maximum charging delay (unit: hour) ($V=10$) of EVs.} 
    \label{tab:delay}
    \begin{tabular}{ccccccccccc}
    \toprule 
    & Group 1 & Group 2 & Group 3 & Group 4 \\ 
    \midrule 
    $w = 1$ &  1.00 & 0.58 & 0.58 & 0.50 \\ 
    $w = 2$ &  1.92 & 1.08 & 0.92 & 0.83 \\ 
    $w = 5$ &  2.83 & 1.58 & 1.83 & 1.00 \\ 
    $w = 8$ &  3.25 & 1.58 & 2.00 & 1.17 \\ 
    $w = 12$ &  3.50 & 1.75 & 2.08 & 1.17 \\ 
    \bottomrule
    \end{tabular} 
\end{table}

The total cost and maximum charging delay of all EVs under different algorithms are illustrated in Fig. \ref{fig:compare}. Figures \ref{fig:cost_slow} and \ref{fig:delay_slow} are the results of the slow charging case, while Figs. \ref{fig:cost} and \ref{fig:delay} are the results of the fast charging case.
It can be observed from Figs. \ref{fig:cost_slow} and \ref{fig:cost} that the proposed algorithms A1 and A2 achieve lower costs than the greedy algorithm and MPC in both cases, which demonstrates the effectiveness of the proposed framework in cost minimization.
% Another finding from case studies is a limitation of MPC that the proposed algorithm does not have, that is, MPC cannot guarantee the satisfaction of charging demand of all EVs if the lookahead window is not sufficiently large ($w=1, 2, 5, 8$), since MPC tends to be inactive before the arrivals of charging demand enters the lookahead window.
% Therefore, only the results of MPC with a large lookahead window ($w=12$) are compared with the proposed algorithm.
A1 is also notably better than standard Lyapunov optimization in both cases, which indicates the benefit of utilizing future information in the proposed framework.
A1 outperforms A2 remarkably since it has more flexibility in tuning the penalty parameters, which can achieve costs extremely close to the offline optimal cost. For example, Fig. \ref{fig:cost} shows that the gap between the total cost of A1 and the offline optimal cost is only 1\% in the fast charging case, while that of A2 is 46\%, both of which are lower than the theoretical bound given by Proposition \ref{prp:optimality} and Proposition \ref{prp:optimality_alt}.
From Figs. \ref{fig:delay_slow} and \ref{fig:delay}, it can be observed that  the average  charging delay of EVs of the proposed algorithm (both A1 and A2) is notably lower than that of MPC, which indicates the advantages of the proposed framework in controlling the charging delay of EVs. The maximum charging delay of the proposed algorithm is also lower than that of MPC in fast charging case, while it is not so advantageous in slow charging case, which is dominated by the extreme case of a few EVs with high demand and  low charging power. The charging delay of the proposed algorithm is larger than that of the standard Lyapunov optimization, which is consistent with the trade-off between cost minimization and quality of service revealed in Section \ref{sec:lya}.
In addition, the maximum charging delay under A1 is higher than that under A2, which is consistent with the trade-off between cost minimization and quality of service revealed in Section \ref{sec:lya}.
The simulation results of all algorithms are summarized in Table \ref{tab:benchmark}.

% \begin{table*}[!htbp]
%     \centering
%     \caption{Cost and delay comparison of algorithms.}
%     \label{tab:benchmark}
%     \begin{tabular}{@{}ccccccccc@{}}
%         \toprule 
%         \multirow{2}{*}{Algorithm} & Total cost    & \multirow{2}{*}{Relative value} & Unit cost & Average delay & Maximum delay \\
%         & (Million USD) & & (USD/kWh) & (h) & (h) \\
%         \midrule 
%         Greedy  & 1.80 & 182\% & 51.7  & 0.18 & 0.42  \\
%         MPC     & 1.81 & 182\% & 51.8  & 9.46 & 12.6  \\
%         Offline & 0.99 & 100\% & 28.5  & 7.36 & 9.67  \\
%         A1      & 1.01 & 101\% & 28.9  & 7.43 & 9.67 \\
%         A2      & 1.44 & 146\% & 41.5  & 2.91 & 7.83  \\
%         \bottomrule
%     \end{tabular}
% \end{table*}

\begin{table*}[!htbp]
    \centering
    \caption{Cost and delay comparison of algorithms ($w = 12$).}
    \label{tab:benchmark}
    \begin{tabular}{@{}cccccccccccccccccccc@{}}
        \toprule 
        & \multicolumn{6}{c}{Case 1 (slow charging)} & \multicolumn{6}{c}{Case 2 (fast charging)} \\
        \cmidrule(lr){2-7} \cmidrule(lr){8-13}
        & Greedy & MPC & Lyapunov & Offline & A1 & A2 & Greedy & MPC & Lyapunov & Offline & A1 & A2 \\
        \midrule
        Total cost (Thousand USD) & 2.11 & 2.07 & 2.07 & 1.38 & 1.60 & 2.00 & 1.80 & 1.81 & 1.68 & 0.992 & 1.01 & 1.44 \\
        Relative cost           & 153\% & 150\% & 150\% & 100\% & 116\% & 145\% & 182\% & 182\% & 170\% & 100\% & 101\% & 146\% \\
        Unit cost (USD/MWh)     & 55.8 & 54.8 & 54.6 & 36.5 & 42.3 & 52.7 & 51.7 & 51.8 & 48.3 & 28.5 & 28.9 & 41.5 \\
        Average delay (h)       & 1.22 & 9.23 & 1.49 & 6.39 & 6.91 & 3.49 & 0.176 & 9.46 & 1.23 & 7.36 & 7.43 & 2.91 \\
        Maximum delay (h)       & 3.33 & 11.2 &  3.67 & 7.08 & 11.2 & 9.92 & 0.417 & 12.6 & 2.50 & 9.67 & 9.67 & 7.83 \\
        \bottomrule
    \end{tabular}
\end{table*}

% \begin{figure}[htbp]
%     \centering
%     \subfigure[]
% {\includegraphics[width=4.2cm]{Figs/soc_ag1_g7.eps} \label{fig:soc}}
%     \subfigure[] {\includegraphics[width =
% 4.2cm]{Figs/qg_hat_ag1.eps} \label{fig:q}}
%     \subfigure[] {\includegraphics[width =
% 4.2cm]{Figs/zg_hat_ag1.eps} \label{fig:z}}
%     \subfigure[] {\includegraphics[width =
% 4.2cm]{Figs/price_ag1.eps} \label{fig:price}}
%     \caption{Results of simulation for one aggregator:
% {\color{black} (a) The SOC traces of three EVs in a group (red
% dash line: requested SOC threshold; yellow dash line: arrival
% time; grey dot line: departure time); (b) The queues of
% charging demands of three groups; (c) The virtual queues of three groups;}
%     (d) The electricity price.}
%     \label{fig:agg}
% \end{figure}

\subsection{Parametric Studies}

The impact of the length of lookahead window $w$, the parameters $V$ and $\alpha$ on the performance of the proposed algorithm are investigated. 
The simulation results on case 2 are presented in Figs. \ref{fig:cost_win}--\ref{fig:delay_win}. The findings in case 1 are similar, thus are omitted for brevity.
The cost traces against different lengths of lookahead window with $V = 10$ are illustrated in Fig. \ref{fig:cost_tr_win_fixed_V}. It can be observed that the cost traces rise more smoothly with a larger length of lookahead window, i.e., the charging strategies become less greedy and more foresighted, which demonstrates the benefit of utilizing future information. 
The costs against different $V$ with each length of lookahead window are illustrated in Fig. \ref{fig:cost_win_V}. Results show that the cost can be reduced by increasing $V$, which is consistent with the theoretical analysis on the drift-plus-penalty algorithm in Section \ref{sec:lya}. The impact of window length on cost reduction is more complicated. Specifically, when $V$ is small, the cost reduction compared to standard Lyapunov optimization is not significant since the stability of virtual queues is prioritized in this case.
However, when $V$ is large enough ($V > 8$ in Fig. \ref{fig:cost_win_V}), the cost under forecast-enhanced algorithm ($w = 5, 8, 12$) is much lower than that under no-forecast algorithm ($w = 1$), which demonstrates the advantage of utilizing future information in cost minimization.
Note that the results with $V > 11$ are not presented for $w = 8$ and $w = 12$ in Fig. \ref{fig:cost_win_V}, since the charging task may fail if $V$ is too large, especially when the lookahead window is large, which is consistent with the trade-off between cost minimization and charging delay revealed in Section \ref{sec:lya}.

\begin{figure}[htbp]
    \centering
    \subfigure[]{
      \includegraphics[width=4.1cm]{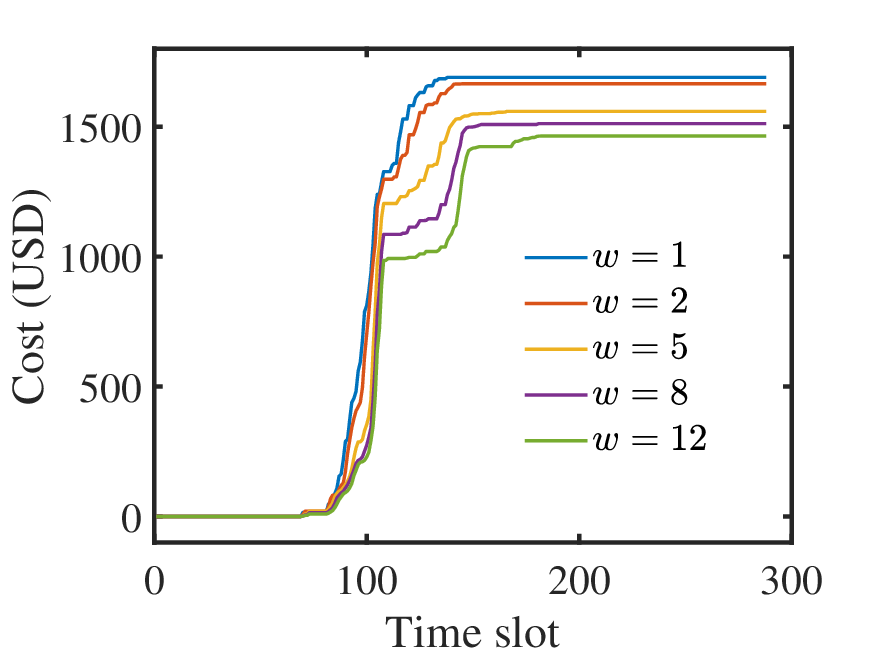} 
      \label{fig:cost_tr_win_fixed_V}
    }
    \subfigure[]{
      \includegraphics[width=4.1cm]{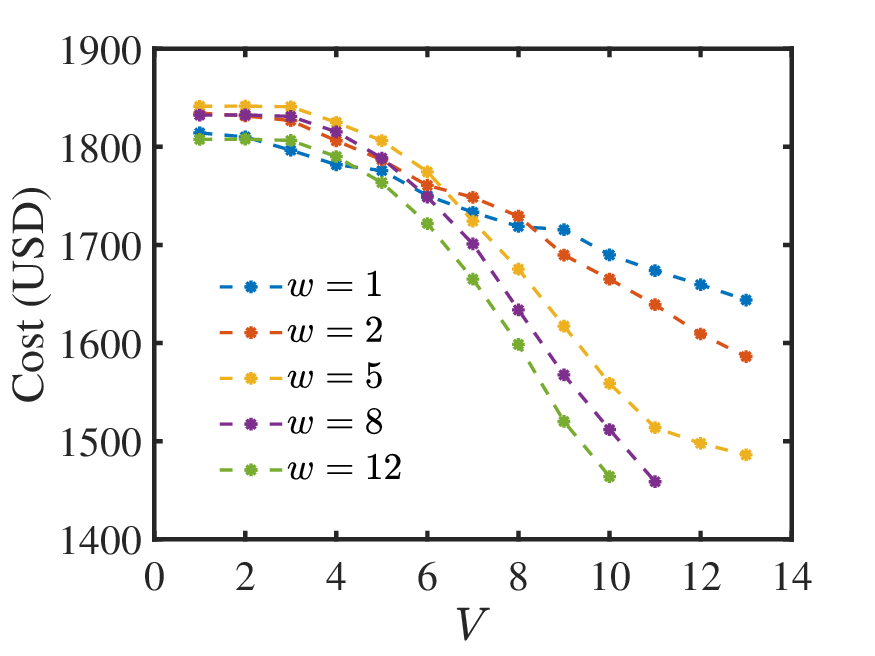}
      \label{fig:cost_win_V}
    }
    \caption{(a) The cost traces against window length ($V = 10$). (b) The costs  against $V$ with each window length.}
    \label{fig:cost_win}
\end{figure}

\begin{figure}[htbp]
    \centering
    \subfigure[]{
      \includegraphics[width=4.1cm]{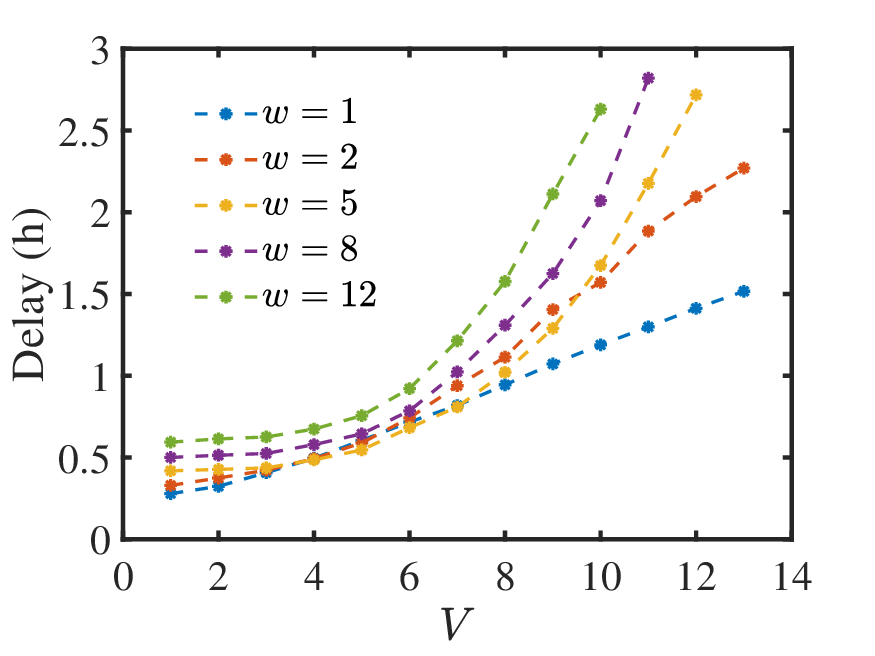}
      \label{fig:delay_mean_win_V}
    }
    \subfigure[]{
      \includegraphics[width=4.1cm]{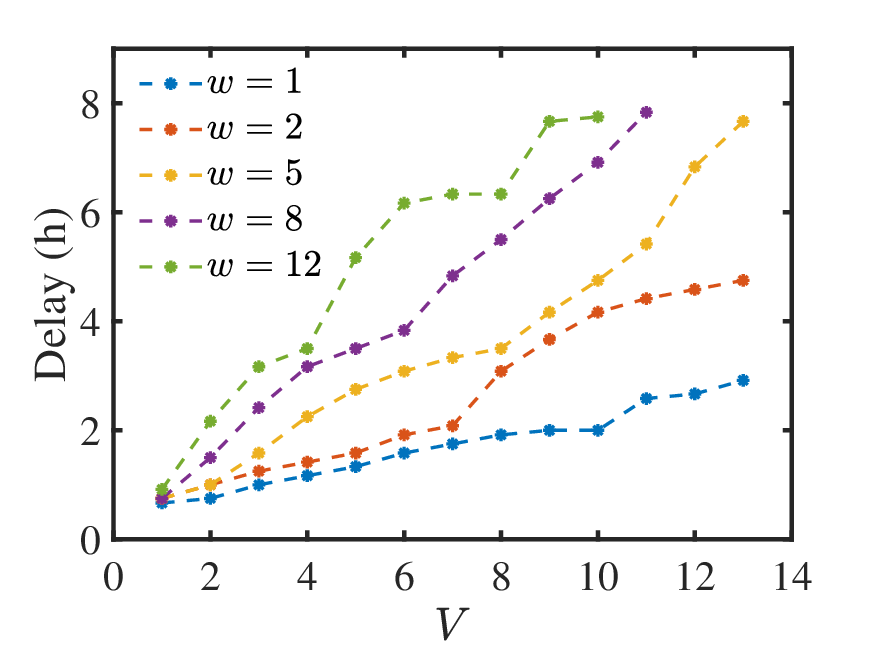}
      \label{fig:delay_max_win_V}
    }
    \caption{(a) Average charging delay against $V$ with each window length. (b) Maximum charging delay against $V$ with each window length.}
    \label{fig:delay_win}
\end{figure}

The average and maximum charging delay against different $V$ with each length of lookahead window are illustrated in Figs. \ref{fig:delay_mean_win_V} and \ref{fig:delay_max_win_V} respectively. It can be observed that charging delay is increased in general as $V$ increases for each $w$ despite some fluctuations in maximum delay, since a larger $V$ puts more weight on cost minimization, leading to longer charging delay. A larger lookahead window also leads to an increase in charging delay, especially when a large $V$ is chosen. This is because a larger lookahead window allows the algorithm to better anticipate future low-price periods, which may result in deferring charging to those periods, thereby increasing the overall charging delay.

It is revealed in Figs. \ref{fig:cost_win_V}, \ref{fig:delay_mean_win_V} and \ref{fig:delay_max_win_V} that a proper lookahead window should be chosen to achieve desirable performance in terms of both efficiency and quality of service. While a larger lookahead window can significantly reduce cost, it may also result in increased charging delay. On the contrary, the benefit of forecasting can be limited if the window is too small (e.g., $w = 2$).
In addition, there exists an optimal $V$ in each case, which achieves the lowest cost while satisfying all charging demand, that is, all EVs have reached their target SOC before departure.
The optimal $V$ against different lengths of lookahead window are presented in Fig. \ref{fig:optimal_V_win}. It shows that the optimal $V$ decreases as the length of lookahead window increases. Another insight from Figs. \ref{fig:cost_win_V} and \ref{fig:optimal_V_win} is that the introduction of future information can effectively simplify the tuning of parameter $V$, since the optimal $V$ does not vary significantly when the length of lookahead window is large enough (e.g., $w \ge 5$ in this case). In other words, any $V$ in a specific range (e.g., $V \in [8, 10]$ in this case) with a moderate length of lookahead window (e.g., $w = 5$ in this case)  achieves satisfactory performance. This is a valuable property in practice because it is usually difficult to precisely find the optimal parameters of an online algorithm in real-world applications.

In addition, the impact of parameter $\alpha$ on the performance of the proposed algorithm is investigated, as shown in Fig. \ref{fig:alpha_win}. Results show that cost is increased and delay is decreased as $\alpha$ is increased for each $w$, which is determined by the role of $\alpha$ in the dynamics of virtual queue $z_g(t)$ in \eqref{eq:zgt}. A larger $\alpha$ leads to higher penalty on charging delay, thus the algorithm tends to reduce the charging delay at the expense of higher cost.

\begin{figure}[t]
  \centering
  \includegraphics[width=5cm]{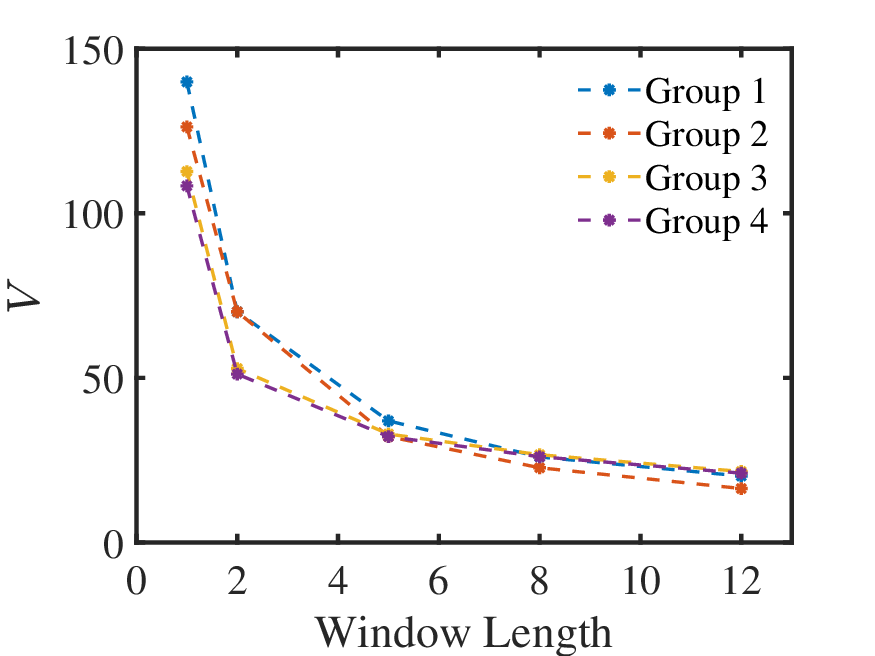}
  \caption{Optimal $V$ against the window length.}
  \label{fig:optimal_V_win}
\end{figure}

\begin{figure}[htbp]
    \centering
    \subfigure[]{
      \includegraphics[width=4.1cm]{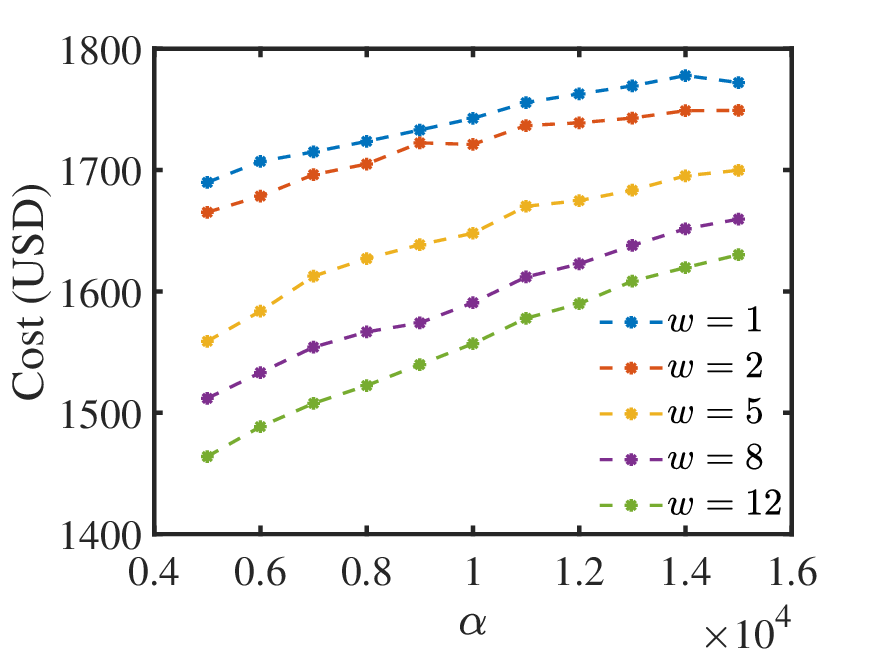}
      \label{fig:cost_win_alpha}
    }
    \subfigure[]{
      \includegraphics[width=4.1cm]{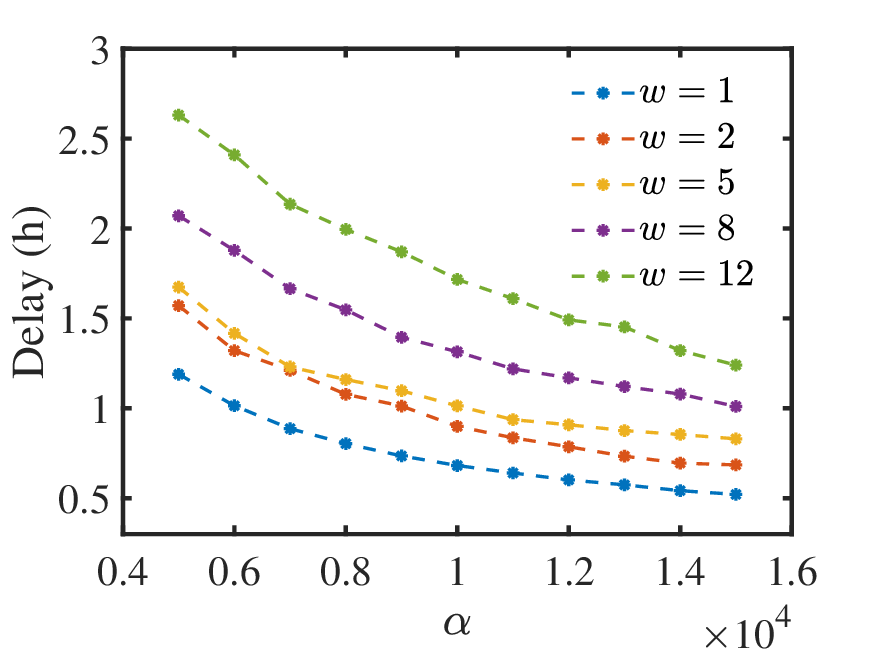}
      \label{fig:delay_win_alpha}
    }
    \caption{(a) The cost traces against $\alpha$ with each window length ($V = 10$). (b) The average charging delay against $\alpha$ with each window length ($V = 10$).}
    \label{fig:alpha_win}
\end{figure}

\subsection{Scalability}

 To evaluate the scalability of the proposed algorithm, various cases are tested with the number of EVs ranging from 50 to 200. The length of lookahead window is set to 12 slots and the average computational time of 50 tests is taken for each case. 
 As shown in Table \ref{tab:scal},  the increase in computational time is not significant as the number of groups increases, which demonstrates the good scalability of the proposed algorithm.
 Specifically,   the computational time is only increased by 1.2 s (6\%) from 50 EVs to 200 EVs. Therefore, the proposed algorithm is capable of real-time implementation for large-scale EV charging scenarios. 

% \begin{figure}[t]
%   \centering
%   \includegraphics[width=4.2cm]{Figs/scalability.eps}
%   \caption{Scalability test.}
%   \label{fig:scal}
% \end{figure}

\begin{table}[t]
    \centering
    \caption{Results of scalability test.} 
    \label{tab:scal}
    \begin{tabular}{ccccccccccc}
    \toprule 
    & Number of EVs & 50 & 100 & 150 & 200 \\
    \midrule 
    & Computational Time (s) & 21.6 & 22.3 & 22.2 & 22.8 \\
    \bottomrule
    \end{tabular} 
\end{table}

\subsection{Robustness to Prediction Error}

The robustness of the proposed algorithm to prediction error is investigated by testing the proposed algorithm with imperfect prediction within the lookahead window.
Specifically, the predicted price at each time slot is generated by adding normal noise to the real electricity price, so that the mean absolute percentage error (MAPE) of the predicted price is 0--15\%. 30 cases were tested for each MAPE value, and the box plots of the cases under each MAPE are illustrated in Figs. \ref{fig:cost_mape} and \ref{fig:delay_mape} respectively.
The plots of algorithms are dodged horizontally for better visualization, i.e., the horizontal position of each box plot is shifted  by a small value so that the plots would not overlap with each other. The plots divided by dashed lines are from the cases with the same MAPE in Fig. \ref{fig:mape}.
Results show that there is a noticeable increase in both cost and charging delay of the proposed algorithm with heterogeneous penalty  (A1) as the MAPE increases from 10\% to 15\%, while no obvious increase is observed in the results of the proposed algorithm with homogeneous penalty (A2). In other words, A2 is more robust to prediction error than A1.
This is because each group has been pushed to its limit, i.e., the optimal $V$ respectively in the heterogeneous setting.
Nonetheless, the proposed algorithm (both A1 and A2) can still achieve better performance  compared to MPC in terms of both cost and delay when prediction error is present, even if MPC has more information (the departure of an EV is forecast when it enters the lookahead window). 

\begin{figure}[htbp]
    \centering
    \subfigure[]{
      \includegraphics[width=\linewidth]{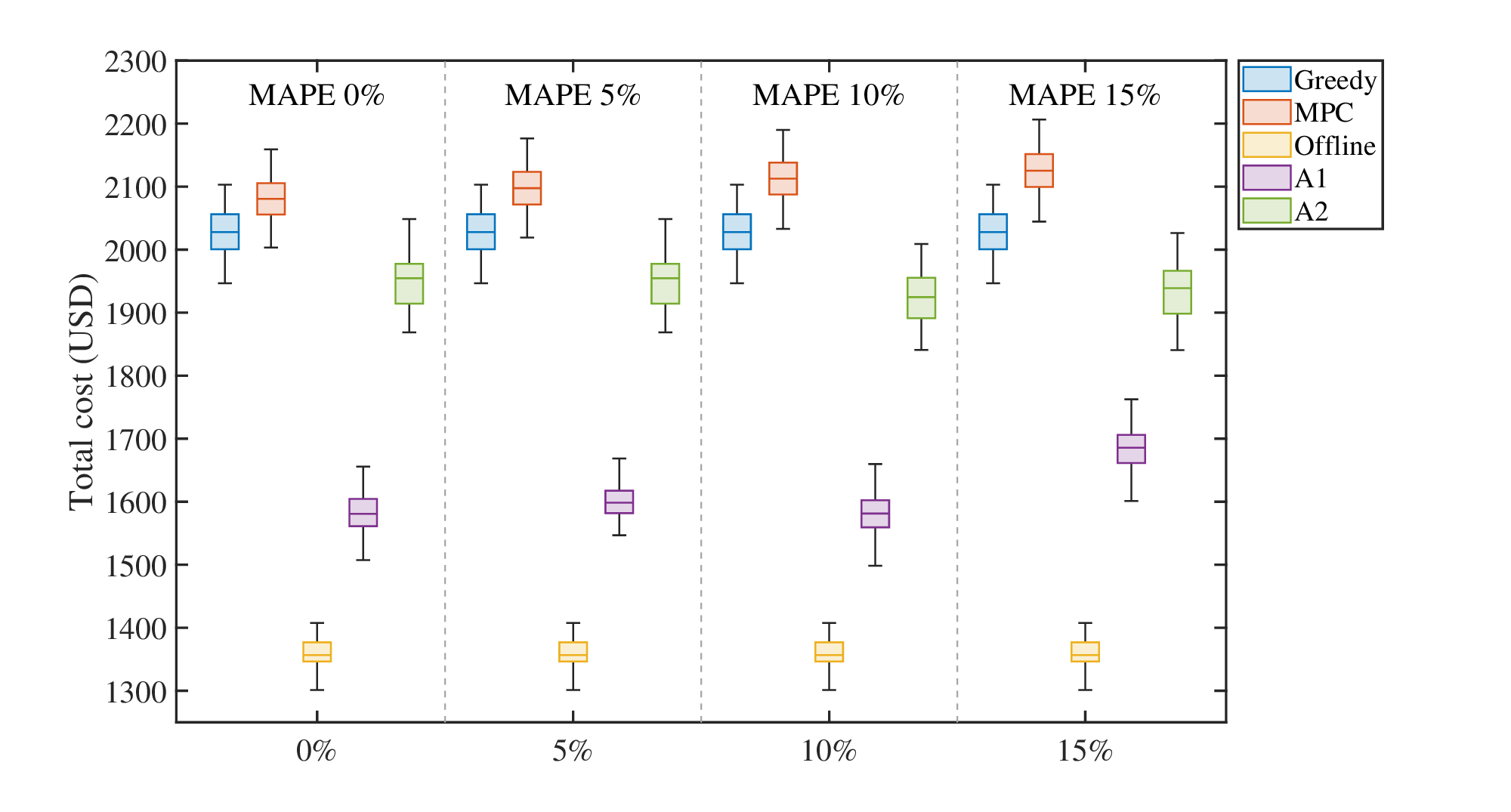}
      \label{fig:cost_mape}
    }
    \subfigure[]{
      \includegraphics[width=\linewidth]{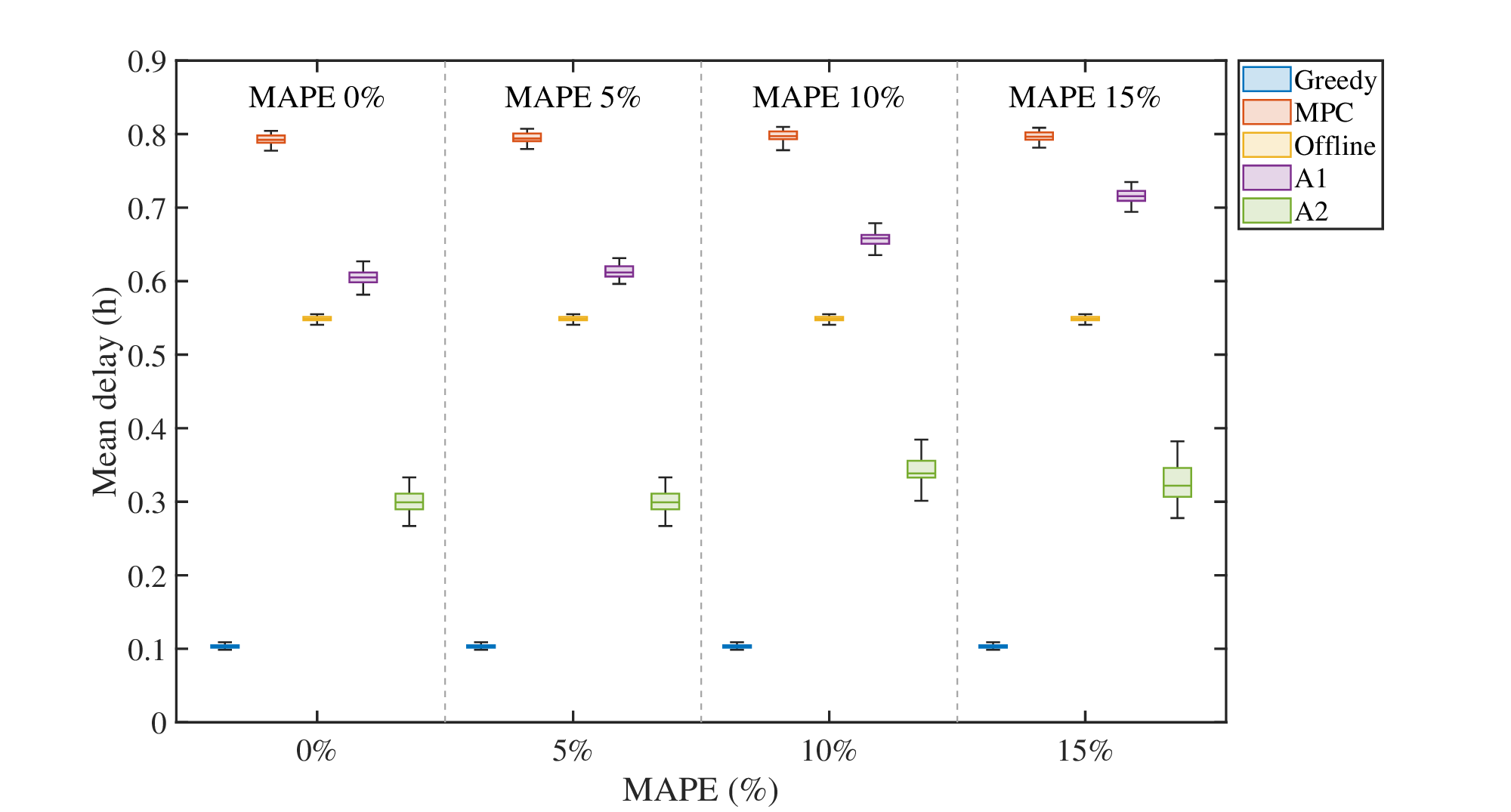}
      \label{fig:delay_mape}
    }
    \caption{Performance of the proposed algorithm against benchmarks with MAPE = 0\%, 5\%, 10\% and 15\%. The plots are dodged horizontally for better visualization. (a) The total cost of all EVs. (b) The mean charging delay of all EVs.}
    \label{fig:mape}
\end{figure}

\section{Conclusion}
\label{sec:conclu} 

In this paper, a novel forecast-enhanced Lyapunov optimization framework is proposed for real-time EV charging scheduling. By integrating future information into the Lyapunov optimization framework, the proposed algorithm can effectively reduce the total cost of EV charging while guaranteeing the quality of service. Theoretical bounds on both optimality and charging delay are derived. Moreover, an extension of the proposed framework with heterogeneous penalty parameters is presented.
Case studies show that the proposed algorithm outperforms benchmark algorithms significantly in terms of cost minimization and charging delay control. The introduction of heterogeneous penalty  can further improve the performance of the proposed algorithm.
Parametric studies reveal the impact of key parameters (e.g., the length of lookahead window, penalty parameter) on the performance of the proposed algorithm, which provide useful guidelines for parameter tuning in practice. The scalability test demonstrates the capability of the proposed algorithm for real-time implementation in large-scale EV charging scenarios.
Future work may focus on the impact of the prediction error on the performance of the proposed algorithm or extending the framework to more complicated scenarios, e.g., considering vehicle-to-grid services.

\section*{Acknowledgment}
The authors would like to thank Professor Steven H. Low of California Institute of Technology for his helpful suggestions and insightful comments on this work.

\bibliographystyle{IEEEtran}
\bibliography{ref}

\clearpage

\appendices \makeatletter \@addtoreset{equation}{section}
\@addtoreset{theorem}{section}
\makeatother
\setcounter{equation}{0}

\counterwithin{equation}{section}
\counterwithin{theorem}{section}

\section{Proof of Proposition \ref{prp:delay}}
\label{appd:delay}

\begin{proof}
  Suppose Proposition \ref{prp:delay} does not hold, i.e., there  exists charging demand $a_{g}(t_{0})$ that cannot be served on or before  the time slot $t_{0}+ D_g$. 
  Since $q_{g}(t)$ and $z_{g}(t)$  are processed       by FIFO method, we have the following conditions:
  \begin{align}
    & q_{g}(t) > 0, ~ \forall t \in [t_{0}, t_{0}+ D_{g}-1],    \label{ineq:qgpos} \\
    & \sum_{t = t_0}^{t_0 + D_{g}-1}x_{g}(t) < Q_{g}. \label{ineq:xq_Qg}
  \end{align}
  Then \eqref{eq:zgt} can be simplified by \eqref{ineq:qgpos}:
  \begin{equation}
    \label{ineq:zgpos}
    z_{g}(t + w ) \ge z_{g}(t) + \frac{\alpha_{g}}{R_{g}} -    \sum_{\tau=0}^{w-1}    x_{g}(t+\tau), ~ \forall t \in [t_{0}, t_{0}+ D_{g}-1].
  \end{equation}
  Suppose $nw \le D_g < (n+1)w$, where $n$ is an integer. By \eqref{ineq:xq_Qg} and $x_g \ge 0$, we have
  \begin{equation}
    \label{ineq:xq_Qg_2}
    \sum_{t = t_0}^{t_0 + nw -1}x_{g}(t) \le \sum_{t = t_0}^{t_0 + D_{g}-1}x_{g}(t) < Q_{g}.
  \end{equation}
  Summing \eqref{ineq:zgpos} over $t_{0}, t_{0}+w, ..., t_{0}+nw$,
  we obtain
  \begin{equation}
    z_{g}(t_{0}+ nw) - z_{g}(t_{0}) \ge \frac{(n+1) \alpha_{g}}{R_{g}}    - \sum_{t = t_0}^{t_0 + nw -1}x_{g}(t).
  \end{equation}
  Obviously, $z_{g}(t) \ge 0, ~ t = 1, 2, \dots, T$. By \eqref{ineq:xq_Qg_2} we have
  \begin{align}
    Z_{g}\ge z_{g}(t_{0}+ nw) - z_{g}(t_{0}) >    \frac{(n+1) \alpha_{g}}{R_{g}}- Q_{g}.
  \end{align}
  Then $D_g$ can be bounded by
  \begin{align}
    D_{g} < (n+1) w < \frac{w R_{g}(Q_{g}+ Z_{g})}{\alpha_{g}},
  \end{align}
  which completes the proof.
\end{proof}

\section{Proof of Proposition \ref{prp:optimality}}
\label{appd:optimality}

\begin{proof}

Take $\hat{x}_{g}^{t}(t+\tau), ~ \tau = 0, \dots, w-1$ into the drift-plus-penalty term \eqref{obj:drift_penalty}, which can be bounded by \eqref{eq:LyaDriftExpan}, i.e.,
we have
\begin{equation}
\label{ineq:cmp_f1}
\begin{aligned}
  & \Delta(\boldsymbol{\Theta}(t)) + V \mathbb{E} \left[ \sum_{\tau=0}^{w-1}  \hat{f}^{(t)}(t+\tau) \right] \\
  \le ~ & B + V \mathbb{E} \left[ \sum_{\tau=0}^{w-1} \hat{f}^{(t)}(t+\tau) \right]         \\
  & - \sum_{g = 1}^{G}\left((q_{g}(t) +  z_{g}(t))\sum_{\tau=0}^{w-1} \hat{x}_{g}^{(t)}(t+\tau) \right) \\
  \le ~ & B + V \sum_{\tau=0}^{w-1}f^{*}(t+\tau)          \\
  & - \sum_{g = 1}^{G}\left((q_{g}(t) +  z_{g}(t))\sum_{\tau=0}^{w-1} x_{g}^{*}(t+\tau) \right),
\end{aligned}
\end{equation}
The second inequality of \eqref{ineq:cmp_f1} holds because $\hat{x}_{g}^{(t)}(t+\tau), ~ \tau = 0, \dots, w-1$ are the optimal solution of \textbf{P3}, whose objective function is  the left-hand side of the second inequality.
Replace $\Delta(\boldsymbol{\Theta}(t))$ by its definition \eqref{eq:LyaDrift}, and sum
\eqref{ineq:cmp_f1} over $t = 1, 2, \dots, T$, yielding
\begin{equation}
\label{ineq:cmp_f2}
\begin{aligned}
  & \mathbb{E} \left[ \sum_{\tau=0}^{w-1}L(\boldsymbol{\Theta}(T+1+\tau)) -    \sum_{\tau=0}^{w-1}L(\boldsymbol{\Theta}(1+\tau))   \right]  \\
  & + V \sum_{t=1}^{T} \mathbb{E}\left[ \sum_{\tau=0}^{w-1}\hat{f}^{(t-\tau)}(t)  \right]  \\
  & + V \mathbb{E} \left[ \sum_{t=T+1}^{T+w-1}      \sum_{\tau=t-T}^{w-1} \hat{f}^{(t-\tau)} (t) -  \sum_{t=1}^{w-1} \sum_{\tau=t}^{w-1} \hat{f}^{(t-\tau)} (t) \right] \\
  \le & ~ BT + w V \sum_{t = 1}^{T}f^{*}(t)  \\
  & + V \left( \sum_{t=T+1}^{T+w-1}      \sum_{\tau=t-T}^{w-1} f^* (t) - \sum_{t=1}^{w-1}  \sum_{\tau=t}^{w-1} f^* (t) \right) 
\end{aligned}
\end{equation}
Recall the assumption in Section \ref{subsec:drift_plus_penalty} that $\hat{f}^{(\tau)}(t) = 0$, $\forall \tau \le 0$, $\forall t$, thus they do not contribute to the sum in \eqref{ineq:cmp_f2}.
Divide \eqref{ineq:cmp_f2} by $w VT$ and let $T\to \infty$,
\begin{equation}
\begin{aligned}
  & f^{\text{ONL}}- f^{\text{OFL}}  \\
  =   & \lim_{T\to \infty}\frac{1}{T}\sum_{t=1}^{T}\left(    \mathbb{E}\left[      \frac{1}{w}\sum_{\tau=0}^{w-1}\hat{f}^{(t-\tau)}(t) \right] - f^{*}(t) \right) \\
  \le & \frac{B}{w V}.
\end{aligned}
\end{equation}
\end{proof}

\section{Bounding the Increment of Virtual Queues with Forecast}

In this appendix, we discuss the increment of backlogs of a queue $q(t)$ with $w$-slot lookahead. The update rule of $q(t)$ is given by:
\begin{equation}
  q(t+w)=\max\left\{q(t)-\sum_{s=0}^{w-1}x(t+s)+\sum_{s=0}^{w-1}a(t+s),\,0\right\}.
\end{equation}
And we assume $q(0)=q(1)=\dots=q(w-1)=0$.

\begin{lemma}
  \label{lem:search_bound}
  If $0\le a(t),x(t)\le X$ for all $t$, then for every integer $t\ge 0$ and every $\tau\in\{0,\dots,w-1\}$ with $t\ge \tau$, we have
  \begin{equation}
    \label{ineq:search_bound}
    q(t-\tau)\ge q(t)-wX.
  \end{equation}
  In particular, if $q(t)\ge M$, then
  \begin{equation}
    q(t-\tau)\ge M-wX.
  \end{equation}
\end{lemma}

\begin{proof}
  Let $d(t)=a(t)-x(t)$, so that $|d(t)|\le X$ for all $t$. Then the recursion can be written as
  \begin{equation}
    q(t+w)=\left[q(t)+\sum_{s=0}^{w-1}d(t+s)\right]^+.
  \end{equation}
  Fix $t\ge 0$ and write $t=nw+r$ with $n\ge 0$ and $r\in\{0,\dots,w-1\}$. By iterating the recursion along residue class $r$, we obtain the representation
  \begin{equation}
    \label{eq:search_representation}
    q(t) = \max_{0\le k\le n}\sum_{i=kw+r}^{t-1} d(i).
  \end{equation}
  Now fix $\tau\in\{0,\dots,w-1\}$ with $t\ge \tau$, and let $u=t-\tau$. We will show that
  \begin{equation}
    \label{ineq:search_bound_alt}
    q(t) \le q(u)+wX,
  \end{equation}
  which implies the desired result.
  We proceed by discussing two cases.

  \noindent
  \textbf{Case 1:} $\tau\le r$.
  In this case, write $u=nw+s$ with $s=r-\tau$. For any $k\in\{0,\dots,n\}$, we decompose
  \begin{equation}
    \sum_{i=kw+r}^{t-1} d(i)    =    \sum_{i=(k+1)w+s}^{u-1} d(i)    +    \sum_{i=kw+r}^{(k+1)w+s-1} d(i)    +    \sum_{i=u}^{t-1} d(i).
  \end{equation}
  The first term is one of the candidates in the maximization formula for $q(u)$ according to \eqref{eq:search_representation}, $\forall k\in\{0,\dots,n-1\}$ . Hence, we have
  \begin{equation}
    \label{ineq:search_bound_case1}
    \sum_{i=(k+1)w+s}^{u-1} d(i) \le q(u), \quad \forall k\in\{0,\dots,n-1\}.
  \end{equation}
  Note that \eqref{ineq:search_bound_case1} holds trivially for $k=n$.
  The second term contains $w-\tau$ terms, and the third term contains $\tau$ terms, so together they contain exactly $w$ terms. Therefore,
  \begin{equation}
  \label{ineq:search_bound_interm}
    \sum_{i=kw+r}^{t-1} d(i) \le q(u)+wX.
  \end{equation}
  Taking the maximum over $k$ yields \eqref{ineq:search_bound_alt}.

  \noindent
  \textbf{Case 2:} $\tau>r$.
  In this case, write $u=(n-1)w+s$ with $n\ge 1$ and $s=w+r-\tau$. For any $k\in\{0,\dots,n\}$, we decompose
  \begin{equation}
    \sum_{i=kw+r}^{t-1} d(i)    =    \sum_{i=kw+s}^{u-1} d(i)    +    \sum_{i=kw+r}^{kw+s-1} d(i)
    +    \sum_{i=u}^{t-1} d(i).
  \end{equation}
  Again, the first term is a candidate for $q(u)$. The second term contains $w-\tau$ terms, and the third term contains $\tau$ terms, so together they contain exactly $w$ terms. Hence, \eqref{ineq:search_bound_interm} holds.
  Taking the maximum over $k$ gives \eqref{ineq:search_bound_alt}.
\end{proof}

\section{Proof of Proposition \ref{prp:queue_bound}}
\label{appd:queue_bound}

\begin{proof}
We prove the bounds by induction at window boundaries. Obviously, $q_g(1)$ and $z_g(1)$ satisfy the bounds.
Assume $q_g(t)\le Q_g$ for some $t$ and we need to show $q_g(t+w)\le Q_g$.
We proceed by discussing two cases.

\noindent
\textbf{Case 1:} $q_g(t) \le V\pi_{\max} + w X_g$.
Then regardless of $x_g$,
\begin{equation}       
    q_g(t+w)    \le    q_g(t) + \sum_{\tau=0}^{w-1} a_g(t+\tau)    \le    V\pi_{\max} + 2wX_g    =    Q_g.
\end{equation}

\noindent
\textbf{Case 2:} $V\pi_{\max} + w X_g < q_g(t) \le V\pi_{\max} + 2w X_g$. 
By Lemma \ref{lem:search_bound}, for any time slot between $t-w$ and $t$, $q_g(t)$ is also bounded below, i.e., 
\begin{equation}
    q_g(t-\tau) > V\pi_{\max}, \quad \tau=0,\dots,w-1.
\end{equation}
Since $z_g(t)\ge 0$, the optimal solutions of \textbf{P3} between $t-w$ and $t$ are given by KKT condition:
\begin{equation}
\label{eq:x_hat_up}
    \hat{x}_g^{(t-\tau)}(t) = X_g, \quad \tau=0,\dots,w-1.
\end{equation}
Hence
\begin{equation}
\label{eq:x_up}
    x_g (t) = \frac{1}{w}        \sum_{\tau=0}^{w-1}\hat{x}_{g}^{(t-\tau)}(t) = X_g.
\end{equation}
\eqref{eq:x_hat_up} and \eqref{eq:x_up} also apply to $x_g (t+1)$, $x_g (t+2)$, ..., $x_g (t+w-1)$, thus
\begin{equation}
    \sum_{\tau=0}^{w-1} x_g(t+\tau) = wX_g.
\end{equation}
Applying the queue update:
\begin{equation}
\begin{aligned}        
    q_g(t+w)   & =    \left[ q_g(t)-wX_g \right]^+    + \sum_{\tau=0}^{w-1} a_g(t+\tau) \\
& \le    V \pi_{\max}   + 2wX_g,
\end{aligned}
\end{equation}
which completes the induction for $q_g$.

The proof of $Z_g$ is similar to that of $Q_g$ and is thus omitted.
\end{proof}

\section{Proof of Proposition \ref{prp:optimality_alt}}
\label{appd:optimality_alt}

\begin{proof}
Replace the scalar $V$ in \eqref{ineq:cmp_f1} by $\boldsymbol{V} = (V_g)_{g = 1, \dots, G}$ and ignore the terms that can be reduced,
\begin{equation}
\label{ineq:drift_alt}
    \Delta(\Theta(t))     + \sum_{g=1}^G V_g \mathbb{E} \left[ \sum_{\tau=0}^{w-1} \hat f_g^{(t)}(t+\tau) \right]    \le    B + \sum_{g=1}^G V_g \sum_{\tau=0}^{w-1} f_g^{*}(t+\tau).
\end{equation}
Sum \eqref{ineq:drift_alt} over \(t=1,\dots,T\), divide by \(wT\), and let $T \to \infty$, yielding
\begin{equation}
\begin{aligned}
     & V_{\min} \left( f^{\text{ONL}}- f^{\text{OFL}} \right) \\
     \le ~ & \sum_{g=1}^G V_g \left( \mathbb{E} \left[ \frac{1}{w} \sum_{\tau=0}^{w-1} \hat f_g^{(t)}(t+\tau) \right] - f_g^{*}(t+\tau) \right) \le \frac{B}{w},
\end{aligned}
\end{equation}
which completes the proof.
\end{proof}

\end{document}